\documentclass[12pt]{amsart}

\usepackage{amssymb}

\usepackage{enumerate}

\usepackage{graphicx}

\makeatletter
\@namedef{subjclassname@2010}{%
  \textup{2010} Mathematics Subject Classification}
\makeatother

\numberwithin{equation}{section}

\usepackage{amsmath,amsthm}
\usepackage{amssymb,latexsym}
\usepackage{enumerate}

\usepackage{amsmath, amsthm, amscd, amsfonts, amssymb, latexsym, graphicx, color}
\usepackage[bookmarksnumbered, colorlinks, plainpages]{hyperref}

\usepackage[cp1250]{inputenc}

\newtheorem{theorem}{Theorem}[section]
\newtheorem{lemma}[theorem]{Lemma}
\newtheorem{proposition}[theorem]{Proposition}
\newtheorem{corollary}[theorem]{Corollary}
\theoremstyle{definition}

\newtheorem{example}[theorem]{Example}

\newtheorem{remark}[theorem]{Remark}
\numberwithin{equation}{section}

\begin{document}


\baselineskip=17pt









\title[Quantifier elimination in quasianalytic structures]
{Quantifier elimination in quasianalytic structures via
non-standard analysis}

\author[K.\ J.\ Nowak]{Krzysztof Jan Nowak}



\begin{abstract}
The paper is a continuation of our earlier article where we
developed a theory of active and non-active infinitesimals and
intended to establish quantifier elimination in quasianalytic
structures. That article, however, did not attain full generality,
which refers to one of its results, namely the theorem on an
active infinitesimal, playing an essential role in our
non-standard analysis. The general case was covered in our
subsequent preprint, which constitutes a basis for the approach
presented here. We also provide a quasianalytic exposition of the
results concerning rectilinearization of terms and of definable
functions from our earlier research. It will be used to
demonstrate a quasianalytic structure corresponding to a
quasianalytic Denjoy--Carleman class which, unlike the classical
analytic structure, does not admit quantifier elimination in the
language of restricted quasianalytic functions augmented merely by
the reciprocal function $1/x$. More precisely, we construct a
plane definable curve, which indicates both that the classical
theorem by J.~Denef and L.~van den Dries as well as
\L{}ojasiewicz's theorem that every subanalytic curve is
semianalytic are no longer true for quasianalytic structures.
Besides rectilinearization of terms, our construction makes use of
some theorems on power substitution for Denjoy--Carleman classes
and on non-extendability of quasianalytic function germs. The last
result relies on Grothendieck's factorization and open mapping
theorems for (LF)-spaces.
\end{abstract}

\maketitle

\section{Introduction}

The paper is a continuation of our earlier article~\cite{Now2}
devoted to quantifier elimination in quasianalytic structures,
where we developed a theory of active and non-active
infinitesimals in the non-standard models of the universal diagram
of a given quasianalytic structure. This theory allowed us to
simultaneously examine the exchange property and valuation
property for terms in the language $\mathcal{L}$ of restricted
quasianalytic functions augmented by rational powers and,
eventually, to establish quantifier elimination and description of
definable functions by $\mathcal{L}$-terms. That article, however,
did not attain full generality, which refers to one of its
results, namely, the theorem on an active infinitesimal ({\em
op.~cit.\/}, Theorem~4.4) playing an essential role in our
non-standard analysis. The general case was covered in our
subsequent preprint~\cite{Now4} which constitutes a basis for the
approach presented here. Recently, quantifier elimination for
generalized quasianalytic algebras was achieved by J.-P.~Rolin and
T.~Servi~\cite{Rol-Ser} by means of other methods, an essential
ingredient of which was a quasianalytic version of
rectilinearization of definable sets.

\vspace{1ex}

We also provide a quasianalytic exposition of the results
concerning rectilinearization of $\mathcal{L}$-terms and of
definable functions from our article~\cite{Now3}, which is a
technique much more delicate in comparison to rectilinearization
of definable sets. Finally, we shall use this technique to
demonstrate a quasianalytic structure corresponding to a
quasianalytic Denjoy--Carleman class which, unlike the classical
analytic structure, does not admit quantifier elimination in the
language of restricted quasianalytic functions augmented merely by
the reciprocal function $1/x$. The construction was given in our
preprints~\cite{Now7}. This indicates that the classical theorem
by J.~Denef and L.~van den Dries~\cite{Denef-Dries} does no longer
hold for quasianalytic structures. More precisely, we construct a
plane definable curve, which also shows that the classical theorem
of \L{}ojasiewicz~\cite{Loj} that every subanalytic curve is
semianalytic is no longer true for quasianalytic structures.
Besides rectilinearization of terms, our construction applies some
theorems on power substitution for Denjoy--Carleman classes and on
non-extendability of quasianalytic function germs.

\vspace{1ex}

The content of this paper falls into eight sections discussed
briefly below. In Section~2, we provide an exposition of results
on rectilinearization of $\mathcal{L}$-terms from our
paper~\cite{Now3}. Section~3 presents several annotated results on
infinitesimals (including the theorem on an active infinitesimal)
from our paper~\cite{Now2} which are ingredients of one induction
procedure, investigated in the next section. Section~5 establishes
quantifier elimination and description of definable functions by
$\mathcal{L}$-terms as well as general versions of the valuation
property and rectilinearization of definable functions.

\vspace{1ex}

The purpose of the subsequent sections is to construct a
counterexample to the aforementioned problem, formulated
in~\cite{Now4}, whether a quasianalytic structure admits
quantifier elimination in the language augmented merely by the
reciprocal function $1/x$. The construction of a counterexample,
given in Section~8, makes use of rectilinearization of terms as
well as of two function-theoretic theorems about Denjoy--Carleman
classes. The first, established in Section~6, is concerned with
power substitution. The other, presented in Section~7, is a
refinement of the non-extendability theorem given by
V.~Thilliez~\cite{Thil-2}, relying on Grothendieck's factorization
and open mapping theorems for (LF)-spaces.

\vspace{1ex}

Let us recall (cf.~\cite{Now1,Now2,Now5,Now6}) that a
quasianalytic structure $\mathbb{R}_{\mathcal{Q}}$ is the
expansion of the real field by restricted $\mathcal{Q}$-analytic
functions (abbreviated to $\mathcal{Q}$-functions) determined by a
system $\mathcal{Q} =(\mathcal{Q}_{m})_{m \in \mathbb{N}}$ of
sheaves of local $\mathbb{R}$-algebras of smooth functions on
$\mathbb{R}^{m}$, subject to the following six conditions:

\begin{enumerate}
  \item each algebra $\mathcal{Q}(U)$ contains the
      restrictions of polynomials;
  \item $\mathcal{Q}$ is closed under composition, i.e.\ the
      composition of $\mathcal{Q}$-maps is a $\mathcal{Q}$-map (whenever it is
      well defined);
  \item $\mathcal{Q}$ is closed under inverse, i.e.\ if
      $\varphi : U \longrightarrow V$ is a $\mathcal{Q}$-map between
      open subsets $U,V \subset \mathbb{R}^{n}$, $a \in U$, $b \in
      V$ and if $\partial \varphi/\partial x (a) \neq 0$, then
      there are neighbourhoods $U_{a}$ and $V_{b}$ of $a$ and
      $b$, respectively, and a $\mathcal{Q}$-diffeomorphism $\psi: V_{b}
      \longrightarrow U_{a}$ such that $\varphi \circ \psi$ is
      the identity map on $V_{b}$;
  \item $\mathcal{Q}$ is closed under differentiation;
  \item $\mathcal{Q}$ is closed under division by a
      coordinate, i.e.\ if a function $f \in \mathcal{Q}(U)$
      vanishes for $x_{i}=a_{i}$, then $f(x)=
      (x_{i}-a_{i})g(x)$ with some $g \in \mathcal{Q}(U)$;
  \item $\mathcal{Q}$ is quasianalytic, i.e.\ if $f \in
      \mathcal{Q}(U)$ and the Taylor series $\widehat{f}_{a}$
      of $f$ at a point $a \in U$ vanishes, then $f$ vanishes
      in the vicinity of $a$.
\end{enumerate}


Note that $\mathcal{Q}$-analytic maps (abbreviated to
$\mathcal{Q}$-maps) give rise, in the ordinary manner, to the
category $\mathcal{Q}$ of $\mathcal{Q}$-manifolds, which is a
subcategory of that of smooth manifolds and smooth maps.
Similarly, $\mathcal{Q}$-analytic, $\mathcal{Q}$-semianalytic and
$\mathcal{Q}$-subanalytic sets can be defined. The above
conditions ensure some (limited) resolution of singularities in
the category $\mathcal{Q}$, including transformation to a normal
crossing by blowing up (cf.~\cite{BM0,BM,Rol-Spei-Wil}), upon
which the geometry of quasianalytic structures relies; especially,
in the absence of their good algebraic properties
(cf.~\cite{Rol-Spei-Wil,Now1,Now2,Now4}). Consequently, the
structure $\mathbb{R}_{\mathcal{Q}}$ is model complete and
o-minimal. Its definable subsets coincide with those subsets of
$\mathbb{R}^{n}$, $n \in \mathbb{N}$, that are
$\mathcal{Q}$-subanalytic in a semialgebraic compactification of
$\mathbb{R}^{n}$. On the other hand, every polynomially bounded,
o-minimal structure $\mathcal R$ determines a quasianalytic system
of sheaves of germs of smooth functions that are locally definable
in $\mathcal R$.

\vspace{1ex}

The examples of such categories are provided by quasianalytic
Denjoy--Carleman classes $\mathcal{Q}_{M}$, where $M=(M_{n})_{n\in
\mathbb{N}}$ are increasing sequences with $M_{0}=1$. The class
$\mathcal{Q}_{M}$ consists of smooth functions $f(x) =
f(x_{1},\ldots,x_{m})$ in $m$ variables, $m \in \mathbb{N}$, which
are locally submitted to the following growth condition for their
derivatives:
$$ \left| \partial^{|\alpha|} f
/\partial x^{\alpha} (x) \right| \leq C \cdot
   R^{|\alpha|} \cdot |\alpha|! \cdot M_{|\alpha|} \ \ \
   \mbox{ for all } \ \ \alpha \in \mathbb{N}^{n}, $$
with some constants $C, R >0 $ depending only on the vicinity of a
given point. This growth condition is often formulated in a
slightly different way:
$$ \left| \partial^{|\alpha|} f
/\partial x^{\alpha} (x) \right| \leq C \cdot
   R^{|\alpha|} \cdot M_{|\alpha|}' \ \ \
   \mbox{ for all } \ \ \alpha \in \mathbb{N}^{n}, $$
where $M_{n}' = n! M_{n}$. Obviously, the class $\mathcal{Q}_{M}$
contains the real analytic functions.

\vspace{1ex}

In order to ensure some important algebraic and analytic
properties of the class $\mathcal{Q}_{M}$, it suffices to assume
that the sequence $M$ or $M'$ is log-convex. The latter implies
that it is closed under multiplication (by virtue of the Leibniz
formula). The former assumption is stronger and implies that it is
closed under composition (Roumieu~\cite{Rou}) and under inverse
(Komatsu~\cite{Kom}); see also~\cite{BM}. Hence the set
$\mathcal{Q}_{m}(M)$ of germs at $0 \in \mathbb{R}^{m}$ of
$\mathcal{Q}_{M}$-analytic functions is a local ring. Then,
moreover, the class $\mathcal{Q}_{M}$ is quasianalytic iff
$$ \sum_{n=0}^{\infty} \, \frac{M_{n}}{(n+1)M_{n+1}} = \infty $$
(the Denjoy--Carleman theorem; see e.g.~\cite{Ru}), and is closed
under differentiation and under division by a coordinate iff
$$ \sup_{n} \, \sqrt[n]{\left( \frac{M_{n+1}}{M_{n}} \right)} <
   \infty $$
(cf.~\cite{M2,Thil}). It is well-known (cf.~\cite{C,C-M,Thil})
that, given two log-convex sequences $M$ and $N$, the inclusion
$\mathcal{Q}_{M} \subset \mathcal{Q}_{N}$ holds iff there is a
constant $C>0$ such that $M_{n} \leq C^{n}N_{n}$ for all $n \in
\mathbb{N}$ or, equivalently,
$$ \sup \, \left\{  \sqrt[n]{\frac{M_{n}}{N_{n}}}: \  n \in \mathbb{N} \right\} < \infty. $$

\section{Rectilinearization of terms}

In our paper~\cite{Now3} we established several results concerning
rectilinearization of functions definable by a Weierstrass system.
It was done in two stages: first, we proved that every definable
function is given piecewise by finitely many terms in the language
augmented by rational powers; next, we proceed with
rectilinearization of terms using transformation to a normal
crossing by blowing up and induction with respect to the
complexity of terms. That second stage can be repeated verbatim
for the case of quasianalytic structures. We begin with suitable
terminology.

\vspace{1ex}

By a quadrant in $\mathbb{R}^{m}$ we mean a subset of
$\mathbb{R}^{m}$ of the form:
$$ \{ x=(x_{1},\ldots,x_{m}) \in \mathbb{R}^{m}: x_{i} =0, x_{j}>0,
   x_{k}<0 \ \ $$
$$ \mbox{ for } \ i \in I_{0}, j \in I_{+}, k \in I_{-} \} , $$
where $\{ I_{0}, I_{+}, I_{-} \}$ is a partition of $\{ 1,\ldots,m
\}$; its trace $Q$ on the cube $[-1,1]^{m}$ shall be called a
bounded quadrant. The interior $\mbox{Int}\,(Q)$ of the quadrant
$Q$ is its trace on the open cube $(-1,1)^{m}$. A bounded closed
quadrant is the closure $\overline{Q}$ of a bounded quadrant $Q$,
i.e.\ a subset of $\mathbb{R}^{m}$ of the form:
$$ \overline{Q} := \{ x \in [-1,1]^{m}: x_{i} =0, x_{j} \geq 0,
x_{k} \leq 0 \ \ \mbox{ for } \ i \in I_{0}, j \in I_{+}, k \in
I_{-} \} . $$

In this section, by a normal crossing on a bounded quadrant $Q$ in
$\mathbb{R}^{m}$ we mean a function $g$ of the form
$$ g(x) = x^{\alpha} \cdot u(x), $$
where $\alpha \in \mathbb{N}^{m}$ and $u$ is a function
$\mathcal{Q}$-analytic near $\overline{Q}$ which vanishes nowhere
on $\overline{Q}$. Below we state the theorem on
rectilinearization of terms (\cite[Theorem~1]{Now3}) for the case
of the quasianalytic structure $\mathbb{R}_\mathcal{Q}$.

\begin{theorem}\label{rect-th} (Simultaneous Rectilinearization of
$\mathcal{L}$-terms) If
$$ f_{1},\ldots,f_{s}: \mathbb{R}^{m} \longrightarrow \mathbb{R} $$
are functions given piecewise by a finite number of
$\mathcal{L}$-terms, and $K$ is a compact subset of
$\mathbb{R}^{m}$, then there exists a finite collection of
modifications
$$ \varphi_{i}: [-1,1]^{m} \longrightarrow \mathbb{R}^{m}, \ \ \ \
   i=1,\ldots,p, $$
such that

1) each $\varphi_{i}$ extends to a $\mathcal{Q}$-map in a
neighbourhood of the cube $[-1,1]^{m}$, which is a composite of
finitely many local blow-ups with smooth $\mathcal{Q}$-analytic
centers and local power substitutions;

2) the union of the images $\varphi_{i}((-1,1)^{m})$,
$i=1,\ldots,p$, is a neighbourhood of $K$;

3) for every bounded quadrant $Q_{j}$, $j=1,\ldots,3^{m}$, the
restriction to $Q_{j}$ of each function $f_{k} \circ \varphi_{i}$,
$k=1,\ldots,s$, $i=1,\ldots,p$, either vanishes or is a normal
crossing or a reciprocal normal crossing on $Q_{j}$.
\end{theorem}

\begin{remark}\label{rem-2.1} Observe that, if the functions
$f_{1},\ldots,f_{s}$ are given piecewise by terms in the language
of restricted $\mathcal{Q}$-analytic functions augmented merely by
the reciprocal function $1/x$, then one can require that the
modifications $\varphi_{i}$, $i=1,\ldots,p$, be composite of
finitely many local blow-ups with smooth $\mathcal{Q}$-analytic
centers.
\end{remark}

We now recall the basic notions linked with our method of
decomposition into special cubes, initiated in~\cite{Now1} and
formulated in terms of special modifications
in~\cite[Theorem~2.1]{Now2}. Unless otherwise stated, we say that
$$ \varphi: (0,1)^{d} \longrightarrow S \subset \mathbb{R}^{m} $$
is a special modification if the set $S$ is piecewise given by
$\mathcal{L}$-terms, $\varphi$ is a $\mathcal{Q}$-map in the
vicinity of $[0,1]^{d}$ which is a diffeomorphism of $(0,1)^{d}$
onto $S$, and the inverse map $\psi$ to this diffeomorphism is
piecewise given by a finite number of $\mathcal{L}$-terms. Then
$S$ is called a special cube with associated diffeomorphism
$\varphi$.

\vspace{1ex}

Every bounded subset $F \subset \mathbb{R}^{m}$ given piecewise by
$\mathcal{L}$-terms is a finite union of special cubes ({\em
op.~cit.\/}, Corollary~2.3). Combined with Theorem~\ref{rect-th},
the method of special cubes allows us to obtain the following

\begin{corollary}\label{rect-cor-1} (Desingularization of
$\mathcal{L}$-terms) If
$$ f_{1},\ldots,f_{s}: \mathbb{R}^{m} \longrightarrow \mathbb{R} $$
are bounded functions given piecewise by a finite number of
$\mathcal{L}$-terms, and $K$ is a compact subset of
$\mathbb{R}^{m}$, then there exists a finite collection of special
modifications
$$ \sigma_{i}: (0,1)^{d_{i}} \longrightarrow S_{i} \subset \mathbb{R}^{m}, \ \ \ \
   i=1,\ldots,p, $$
such that

1) the union of the special cubes $S_{i}$, $i=1,\ldots,p$, is a
neighbourhood of $K$;

2) each composite function $f_{k} \circ \varphi_{i}$,
$k=1,\ldots,s$, extends to a $\mathcal{Q}$-function in the
vicinity of $[0,1]^{d_{i}}$.
\end{corollary}

\begin{proof}
We proceed with induction with respect to the dimension $m$. The
assertion is trivial if $m=0$. Supposing the assertion to hold for
$d<m$, we shall prove it for $m$. Take the modifications
$$ \varphi_{i}: [-1,1]^{m} \longrightarrow \mathbb{R}^{m}, \ \ \ \
   i=1,\ldots,p, $$
achieved in Theorem~\ref{rect-th}. For each $i=1,\ldots,p$, there
exists a closed nowhere dense subset $V_{i} \subset
\mathbb{R}^{m}$ given piecewise by $\mathcal{L}$-terms such that
the restriction
$$ \varphi_{i}: (-1,1)^{m} \setminus \varphi_{i}^{-1}(V_{i})
   \longrightarrow \mathbb{R}^{m} \setminus V_{i} $$
is a diffeomorphism onto the image given piecewise by
$\mathcal{L}$-terms, and the inverse map to this diffeomorphism is
piecewise given by a finite number of $\mathcal{L}$-terms. By
decomposition into special cubes~\cite[Theorem~2.1]{Now2}), it is
enough to consider the restrictions of the functions
$f_{1},\ldots,f_{s}$, to the subsets $V_{i}$ of dimensions
$d_{i}<m$, $i=1,\ldots,p$. Again, application of special
modifications reduces the problem to the dimensions $d_{i} <m$,
which finishes the proof by the induction hypothesis.
\end{proof}

Now, let $T$ be the universal diagram of the structure
$\mathbb{R}_{\mathcal{Q}}$ in the language $\mathcal{L}$ of
restricted quasianalytic functions augmented by rational powers,
i.e.\ the set of all universal $\mathcal{L}$-sentences that are
true in $\mathbb{R}_{\mathcal{Q}}$). Fix a model $\mathcal{R}$ of
the universal theory $T$ in the language $\mathcal{L}$. Every
$\mathcal{L}$-substructure of $\mathcal{R}$ is a model of $T$. We
always regard the standard model $\mathbb{R}_{\mathcal{Q}}$ as a
substructure of $\mathcal{R}$. Since the decompositions into
special cubes are described by $\mathcal{L}$-terms (both a special
modification $\varphi$ and its inverse $\psi$), they are preserved
by passage to any model $\mathcal{R}$ of $T$:
$$ F = \bigcup_{j} S_{j} \ \ \Longrightarrow \ \
   F^{\mathcal{R}} = \bigcup_{j} S_{j}^{\mathcal{R}}. $$
For simplicity of notation, we shall usually omit the superscript
$^{\mathcal{R}}$ referring to the interpretations in a model
$\mathcal{R}$, which will not lead to confusion.

\vspace{1ex}

We say that infinitesimals
$\lambda=(\lambda_{1},\ldots,\lambda_{m}) \in \mathcal{R}$ are
analytically dependent, if $\lambda$ lie in a special cube
$S=\varphi((0,1)^{d})$ with $d<m$. We call infinitesimals
$\lambda$ analytically independent if they are not analytically
dependent. Analytical independence is preserved, of course, under
permutation of infinitesimals. We say that a subset $A$ in
$\mathcal{R}$ is analytically independent if every finite subset
$A$ in $\mathcal{R}$ consists of analytically independent
infinitesimals. If $A \subset B$ and the set $B$ is analytically
independent, so is $A$.

\vspace{1ex}

For a subset $A \subset \mathcal{R}$, let $\langle A \rangle$
denote the substructure of $\mathcal{R}$ generated by $A$. We
shall see in the next two sections that span operation $\langle A
\rangle$ satisfies the exchange property.

\vspace{1ex}

The convex hull of $\mathbb{R}$ in $\mathcal{R}$ is a valuation
ring $V$ of bounded (with respect to $\mathbb{R}$) elements in
$\mathcal{R}$; its maximal ideal $\mathfrak{m}$ consists of all
infinitesimals in $\mathcal{R}$. The valuation $v$ induced by $V$
is called the standard valuation on the field $\mathcal{R}$; its
value group $\Gamma_{\mathcal{R}}$ is a $\mathbb{Q}$-vector space.
In order to investigate the valuation $v$, we have established in
paper~\cite{Now2} several results about $\mathcal{Q}$-functions,
recalled in this and the next section.

\vspace{1ex}

Now we are going to formulate rectilinearization of terms in the
language of infinitesimals. Clearly, Corollary~\ref{rect-cor-1}
immediately yields the version stated below, which coincides with
that from~{\em op.~cit.\/}, Corollary~2.6, whose proof, however,
was provided too scantily.

\begin{corollary}\label{rect-cor-2}
Consider an $\mathcal{L}$-term $t(x)$ and positive analytically
independent infinitesimals $\lambda =
(\lambda_{1},\ldots,\lambda_{m})$. If $t(\lambda)$ is bounded,
then there exist a special modification
$$ \varphi: (0,1)^{m} \longrightarrow \mathbb{R}^{m} \ \ \mbox{ with } \ \
\lambda = \varphi (\lambda') \ \mbox{ for some } \ \lambda' \in
(0,1)^{m} $$  such that the superposition $f:= t \circ \varphi$
extends to a $\mathcal{Q}$-function in the vicinity of
$[0,1]^{m}$; in particular we have $t(\lambda) = f(\lambda')$.
\end{corollary}

\section{Some results on active and non-active infinitesimals}

In paper~\cite{Now2}, we developed a theory of active and
non-active infinitesimals, crucial for our approach to the
geometry of quasianalytic structures. Here, we wish to present
several annotated results of this theory, which are ingredients of
an induction procedure investigated in the next section. We say
that an infinitesimal $\mu$ is non-active over infinitesimals
$\lambda=(\lambda_{1},\ldots,\lambda_{m})$ if for each
$\mathcal{L}$-term $t(x)$ we have
$$ v(\mu - t(\lambda)) \in \Gamma_{\langle
\lambda \rangle}. $$ Otherwise, the infinitesimal $\mu$ is called
active over $\lambda$. It is clear that if $\mu$ is non-active
over $\lambda$, so is the infinitesimal $\mu' = s(\lambda)\mu
+t(\lambda)$ that is the value at $(\lambda,\mu)$ of any
$y$-linear $\mathcal{L}$-term.

\vspace{1ex}

We first examined, by means of successive lowering order of a
given $\mathcal{Q}$-function ({\em op.~cit.\/}, Theorem~3.1), its
behaviour at non-active infinitesimals ({\em op.~cit.\/},
Propositions~3.2 and~3.3), and thence conclude that, given a
finite number of infinitesimals
$\lambda=(\lambda_{1},\ldots,\lambda_{m})$, the value group
$\Gamma_{\langle \lambda \rangle}$ is a vector space over
$\mathbb{Q}$ of dimension $\leq m$ ({\em op.~cit.\/},
Corollary~3.4). Also examined were the behaviour of
$\mathcal{L}$-terms ({\em op.~cit.\/}, Proposition~3.5) and the
exchange property ({\em op.~cit.\/}, Proposition~3.6) at such
infinitesimals. However, it was yet more difficult to describe the
behaviour of $\mathcal{L}$-terms at active infinitesimals. To this
end, we introduced the concept of a regular sequence of
infinitesimals: a sequence $\lambda
=(\lambda_{1},\ldots,\lambda_{m})$ of infinitesimals shall be
called regular with main part $\lambda_{1},\ldots,\lambda_{k}$, if
the valuations
$$ v(\lambda_{1}),\ldots,v(\lambda_{k}) \in \Gamma_{\langle \lambda
\rangle} $$ form a basis over $\mathbb{Q}$ of the valuation group
$\Gamma_{\langle \lambda_{1},\ldots,\lambda_{m} \rangle}$.

\vspace{1ex}

Now we list five results from our paper~\cite{Now2}, including the
theorem on an active infinitesimal. Yet it has not been proven in
full generality, nevertheless, covering the classical analytic
case. Their formulations below take into account the number of
infinitesimals under study, because these results will be
encompassed by one induction procedure.

\vspace{2ex}


$\mathbf{(I_{m})}$ \  {\em (op.~cit., Theorem~4.4 on an Active
Infinitesimal)}

\vspace{1ex}

\begin{em}
For any $n \leq m$, consider a regular sequence
$\mu,\lambda_{1},\ldots,\lambda_{n}$ of infinitesimals with main
part $\mu,\lambda_{1},\ldots,\lambda_{k}$ and an
$\mathcal{L}$-term $t(y,x)$, $x=(x_{1}\ldots,x_{n})$, such that
$$ \nu :=t(\mu,\lambda) \not \in \langle \lambda \rangle $$
is an infinitesimal. If $v(\mu) \not \in \Gamma_{\langle \lambda
\rangle}$, then $\nu$ is active over the infinitesimals $\lambda$.
\end{em}

\vspace{2ex}

$\mathbf{(II_{m})}$ \ {\em (op.~cit., Proposition~4.7)}

\vspace{1ex}

\begin{em}
For any $n \leq m$, consider a regular sequence
$\lambda_{1},\ldots,\lambda_{n}$ of infinitesimals with main part
$\lambda_{1},\ldots,\lambda_{k}$ and an infinitesimal $\mu$ with
$v(\mu) \not \in \Gamma_{\langle \lambda \rangle}$. Then $\dim
\Gamma_{\langle \mu, \lambda \rangle} = k+1$ whence
$\mu,\lambda_{1},\ldots,\lambda_{n}$ is a regular sequence of
infinitesimals with main part
$\mu,\lambda_{1},\ldots,\lambda_{k}$.
\end{em}

\vspace{2ex}

$\mathbf{(III_{m})}$ \ {\em (op.~cit., Corollary~4.8; Valuation
Property for $\mathcal{L}$-terms)}

\vspace{1ex}

\begin{em}
\noindent For any $n \leq m$ and infinitesimals
$\mu,\lambda_{1},\ldots,\lambda_{n}$, we have the following
dichotomy:

$\bullet$ either $\mu$ is non-active over $\lambda$, and then
$\Gamma_{\langle \lambda,\mu \rangle} = \Gamma_{\langle \lambda
\rangle}$;

$\bullet$  or $\mu$ is active over $\lambda$, and then $\dim
\Gamma_{\langle \lambda,\mu \rangle} = \dim \Gamma_{\langle
\lambda \rangle} +1$.

\noindent In the latter case, one can find an $\mathcal{L}$-term
$t(x)$ such that
\end{em}
$$ v(\mu - t(\lambda)) \not \in \Gamma_{\langle \lambda
\rangle} \ \ \mbox{ and } \ \ \Gamma_{\langle \lambda, \mu
\rangle} = \Gamma_{\langle \lambda \rangle} \oplus \mathbb{Q}
\cdot v(\mu - t(\lambda)). $$

\vspace{1ex}

$\mathbf{(IV_{m})}$ \ {\em (op.~cit., Corollary~4.9; Exchange
Property for $\mathcal{L}$-terms)}

\vspace{1ex}

\begin{em}
\noindent For any $n \leq m$ and infinitesimals
$\mu,\lambda_{1},\ldots,\lambda_{n}$, if $\nu \in \langle \lambda,
\mu \rangle$ and $\nu \not \in \langle \lambda \rangle$, then $\mu
\in \langle \lambda, \nu \rangle$.
\end{em}

\vspace{2ex}

$\mathbf{(V_{m})}$ \ {\em (op.~cit., Proposition~5.1 and
Corollary~5.2; Behaviour of Analytically Independent
Infinitesimals)}

\vspace{1ex}

\begin{em}
For any $n \leq m$ and any two sets of analytically independent
infinitesimals $\lambda_{1},\ldots,\lambda_{n}$ and
$\lambda_{1}',\ldots,\lambda_{n}'$, if $\langle \lambda \rangle
\subset \langle \lambda' \rangle$, then $\langle \lambda \rangle =
\langle \lambda' \rangle$. Consequently, the infinitesimals
$(\lambda_{1},\ldots,\lambda_{n},\mu)$ are analytically
independent iff $\mu \not \in \langle \lambda \rangle$.
\end{em}

\vspace{2ex}

The proofs of the above results provided in paper~\cite{Now2}
indicate the following inferences:
$$ \mathbf{\mathbf{(I_{m})}  \Longrightarrow \mathbf{(II_{m+1})}, \ \
   (II_{m})}  \Longrightarrow \mathbf{(III_{m})} $$
and
$$ \mathbf{(I_{m})} \wedge
   \mathbf{(III_{m})} \Longrightarrow \mathbf{(IV_{m})}
   \Longrightarrow \mathbf{(V_{m+1})}. $$
Therefore the induction hypothesis $\mathbf{(I_{m-1})}$ implies
the hypotheses $\mathbf{(II_{m})}$, $\mathbf{(III_{m})}$,
$\mathbf{(IV_{m-1})}$ and $\mathbf{(V_{m})}$. Actually, in the
proof we shall make use of the assertions $\mathbf{(I_{m-1})}$,
$\mathbf{(III_{m})}$ and $\mathbf{(V_{m})}$.

\begin{remark}\label{rem-span}
The exchange property for $\mathcal{L}$-terms means exactly that
the structure $\mathcal{R}$ with span operation is geometric, and
that we have at our disposal the concepts of rank and basis for
its substructures.
\end{remark}

\section{Proof of the theorem on an active infinitesimal}

We proceed with induction with respect to the number $m$ of
infinitesimals $\lambda_{1},\ldots,\lambda_{m}$. When $m=0$, then
$\langle \lambda \rangle =\langle \emptyset \rangle= \mathbb{R}$
and the conclusion is evident. So take $m>0$ and assume the
theorem holds when the number of infinitesimals $\lambda$ is
smaller than $m$.

\vspace{1ex}

In~\cite{Now2}, we have reduced the problem to the case where
$$  t(u,v,\tilde{x}) = f(u,v/u,\tilde{x}), \ \ \
    \nu :=t(\mu,\lambda) = f(\mu,\lambda_{1}/\mu,\tilde{\lambda}),
$$
where $f(u,v,\tilde{x})$ is a function $\mathcal{Q}$-analytic at
$0 \in \mathbb{R}^{m+1}$; here $v = x_{1}$,
$\tilde{x}=(x_{2},\ldots,x_{m})$ and $\tilde{\lambda}=
(\lambda_{2}\ldots,\lambda_{m})$.

\vspace{1ex}

The valuation group $\Gamma_{\langle \lambda,\mu \rangle}$ is a
vector space over $\mathbb{Q}$ of dimension $\leq (m+1)$ ({\em
op.cit.}, Corollary~3.4). It is a direct sum of finitely many
archimedean subgroups
$$ \Gamma_{\langle \lambda,\mu \rangle} = G_{1} \oplus \ldots \oplus
   G_{r} \ \ \mbox{ with } \ \  G_{1}^{+} > \ldots > G_{r}^{+}, $$
where $G_{i}^{+}$ stands for the semigroup of all positive
elements of $G_{i}$. It is well-known that every archimedean
ordered abelian group is isomorphic to a subgroup of the ordered
additive group $\mathbb{R}$ of real numbers.

\vspace{1ex}

Let $\varepsilon=(\varepsilon_{1},\ldots,\varepsilon_{p})$ denote
those infinitesimals from among $\lambda$ for which
$$ v(\varepsilon_{1}),\dots,v(\varepsilon_{p}) > G_{2} \oplus \ldots \oplus G_{r} $$
and $\delta=(\delta_{1},\ldots,\delta_{q})$ the remaining
$\lambda$'s; obviously, $p+q=m$. The valued field $\langle
\lambda,\mu \rangle$ can be completed with respect to the standard
valuation $v$, and the completion has the same valuation group.
Notice that the topology induced by $v$ is metrizable with the
basis of zero neighbourhoods consisting of sets of the form
$$ \{ t(\mu,\delta,\varepsilon): \; v(t(\mu,\delta,\varepsilon)) >
   \gamma \}, \ \ \ \gamma \in G_{1}. $$
In the completion, one can deal with formal power series in the
infinitesimals $\varepsilon$ with $\mathcal{Q}$-analytic
coefficients taken on the infinitesimals $\delta$.

\vspace{1ex}

We encounter two cases:

\vspace{1ex}

\begin{em}
Case A, where $\lambda_{1}$ is one of the $\varepsilon$'s, say
$\lambda_{1}=\varepsilon_{1}$;

\vspace{1ex}

\noindent {\em or}

\vspace{1ex}

Case B, where $\lambda_{1}$ is one of the $\delta$'s, say
$\lambda_{1}=\delta_{1}$.
\end{em}

\vspace{2ex}

{\bf CASE A.} Consider the Taylor coefficients
$$ \frac{1}{i!j!} \cdot \frac{\partial^{i+j}f}{\partial u^{i} \partial v^{j}}
   (0,0,\tilde{x}) =: a_{ij}(\tilde{x}), \ \ \ i,j \in \mathbb{N}, $$
which are $\mathcal{Q}$-analytic functions at zero. A crucial role
is played by the following


\begin{lemma}\label{proof-lem-1} Under the assumptions of the
theorem on an active infinitesimal \ $\mathbf{(I_{m})}$, we must
have
$$ \sum_{i=0}^{\infty} \, a_{i,i+s}(\tilde{\lambda}) \cdot
   \varepsilon_{1}^{i} \neq 0 \ \ \mbox{ or } \ \
   \sum_{j=0}^{\infty} \, a_{j+s,j}(\tilde{\lambda}) \cdot
   \varepsilon_{1}^{j} \neq 0 $$
for some $s \in \mathbb{N} \setminus \{ 0 \}$.
\end{lemma}


\begin{proof}
Suppose Lemma~\ref{proof-lem-1} were false. Then
$$ \sum_{i=0}^{\infty} \, a_{i,i+s}(\tilde{\lambda}) \cdot
   \varepsilon_{1}^{i} = 0 \ \ \mbox{ and } \ \
   \sum_{j=0}^{\infty} \, a_{j+s,j}(\tilde{\lambda}) \cdot
   \varepsilon_{1}^{j} = 0 $$
for all $s \in \mathbb{N} \setminus \{ 0 \}$.

\vspace{1ex}

But we can find a model of the universal diagram $T$ with the
infinitesimals $\lambda$ and an infinitesimal $\mu^{*}$ such that
$\lambda,\mu^{*}$ are analytically independent and
$$ v(\mu^{*}),v(\varepsilon_{1}/\mu^{*}) > G_{2} \oplus \ldots \oplus
   G_{r}. $$
Indeed, it follows from the induction hypothesis that the
assertion $\mathbf{(V_{m})}$ holds. The infinitesimals $\lambda$
and $\mu^{*}$ are therefore analytically independent iff $\mu^{*}
\not \in \langle \lambda \rangle$. Consequently, it suffices to
find a model of the universal diagram $T$ along with the diagram
of the structure $\langle \lambda \rangle$ and the sentences of
the form
$$ c \neq t(\lambda) \ \ \ \mbox{ where } \ \ t(x) \ \mbox{ are } \
   \mathcal{L}\mbox{-terms} $$
and the sentence
$$ \frac{1}{2} \, \sqrt{\varepsilon_{1}} < c < 2 \, \sqrt{\varepsilon_{1}} \, ; $$
here $c$ denotes a new constant construed as $\mu^{*}$. Its
existence can be immediately deduced through model-theoretic
compactness.

\vspace{1ex}

Under the conditions stated above, we have
$$ t(\mu^{*},\lambda) = f(\mu^{*},\lambda_{1}/\mu^{*},\tilde{\lambda}) =
   \sum_{i=0}^{\infty} \, a_{i,i}(\tilde{\lambda}) \cdot
   \varepsilon_{1}^{i} + $$
$$ + \sum_{s=1}^{\infty} (\varepsilon_{1}/\mu^{*})^{s}
   \cdot \sum_{i=0}^{\infty} \, a_{i,i+s}(\tilde{\lambda})
   \cdot \varepsilon_{1}^{i} +
   \sum_{s=1}^{\infty} (\mu^{*})^{s} \cdot
   \sum_{j=0}^{\infty} \, a_{j+s,j}(\tilde{\lambda}) \cdot
   \varepsilon_{1}^{i} = $$
$$ = \sum_{i=0}^{\infty} \, a_{i,i}(\tilde{\lambda})
   \cdot \varepsilon_{1}^{i}, $$
and thus we get
$$ \partial t/\partial u \, (\mu^{*},\lambda) = 0. $$
Here
$$ \partial t/\partial u \, (u,x) = \partial /\partial u \, f(u,x_{1}/u,\tilde{x}) $$
is a $\mathcal{Q}$-function on an open set
$$ \Omega = S \times \{ \tilde{x} \in \mathbb{R}^{m-1}: \, |x_{2}|,
   \ldots, |x_{m}| < \rho \}, $$
where $\rho \in \mathbb{R}$, $\rho >0$ is small enough, and
$$ S = \{ (u,x_{1}) \in \mathbb{R}^{2}: \, |u| < \rho, \,
   |x_{1}|<\rho|u| \} $$
is a sector. Since the infinitesimals $\mu^{*},\lambda$ are
analytically independent, the above $\mathcal{Q}$-function must
vanish on an open special cube contained in $\Omega$. By the
identity principle for quasianalytic functions, it vanishes
identically on $\Omega$, and thus
$$ t(u,x) = f(u,x_{1}/u,\tilde{x}) = g(x), $$
where $r \in \mathbb{R}$, $r \in (0,\rho)$, and $ g(x) :=
f(r,x_{1}/r,\tilde{x})$ is a $\mathcal{Q}$-function at $0 \in
\mathbb{R}^{m}_{x}$. Hence
$$ \nu :=t(\mu,\lambda) = g(\lambda) \in \langle \lambda \rangle , $$
and this contradiction completes the proof of
Lemma~\ref{proof-lem-1}.
\end{proof}


Since $\mu \cdot \varepsilon_{1}/\mu = \varepsilon_{1}$, we see
that
$$ v(\mu) \geq 1/2 \: v(\varepsilon_{1}) \ \ \mbox{ or } \ \
   v(\varepsilon_{1}/\mu) \geq 1/2 \: v(\varepsilon_{1}). $$
By symmetry, we may assume that the former condition holds. Then
the series
$$ \nu_{0} := \sum_{i=j=0}^{\infty} \ a_{ij}(\tilde{\lambda}) \cdot
   \mu^{i} \cdot \left( \frac{\varepsilon_{1}}{\mu} \right)^{j} =
   \sum_{i=0}^{\infty} \ a_{ii}(\tilde{\lambda}) \cdot
   \varepsilon_{1}^{i},  $$
$$ \nu_{+} := \sum_{i=0}^{\infty} \sum_{j<i} \ a_{ij}(\tilde{\lambda}) \cdot
   \mu^{i} \cdot \left( \frac{\varepsilon_{1}}{\mu} \right)^{j} =
   \sum_{i=0}^{\infty} \sum_{j<i} \ a_{ij}(\tilde{\lambda})\cdot \mu^{i-j}
   \cdot \varepsilon_{1}^{j} $$
and
$$ \nu_{-} := \sum_{i=0}^{\infty} \left( \frac{1}{i!} \cdot
   \frac{\partial^{i}f}{\partial u^{i}}(0,\frac{\varepsilon_{1}}{\mu},\tilde{\lambda})
   \cdot \mu^{i}
   - \sum_{j\leq i} \ a_{ij}(\tilde{\lambda}) \cdot
   \mu^{i} \cdot \left( \frac{\varepsilon_{1}}{\mu} \right)^{j} \right)
   = $$
 $$  = \sum_{i=0}^{\infty} \left( \frac{1}{i!} \cdot
     \frac{\partial^{i}f}{\partial
     u^{i}}(0,\frac{\varepsilon_{1}}{\mu},\tilde{\lambda})
     - \sum_{j\leq i} \ a_{ij}(\tilde{\lambda}) \cdot
     \left( \frac{\varepsilon_{1}}{\mu} \right)^{j} \right)
     \cdot \mu^{i} $$
are well defined elements of the completion of the valued field
$\langle \lambda,\mu \rangle$ with respect to the standard
valuation $v$. We have $\nu = \nu_{+} + \nu_{0} + \nu_{-}$.
Observe that the $\mathcal{Q}$-analytic function which occurs in
the $i$-th summand of the last series is of the form
$$ \left( \frac{1}{i!} \cdot
     \frac{\partial^{i}f}{\partial
     u^{i}}(0,v,\tilde{x})
     - \sum_{j\leq i} \ a_{ij}(\tilde{x}) \cdot
     v^{j} \right) \cdot u^{i} = g_{i}(v,\tilde{x}) \cdot v^{i+1} \cdot u^{i} $$
 for some function $g_{i}$ $\mathcal{Q}$-analytic at zero.


\begin{lemma}\label{proof-lem-2}
If
$$ \sum_{i=0}^{\infty} \, a_{i,i+s}(\tilde{\lambda}) \cdot
   \varepsilon_{1}^{i} \neq 0 \ \ \mbox{ or } \ \
   \sum_{j=0}^{\infty} \, a_{j+s,j}(\tilde{\lambda}) \cdot
   \varepsilon_{1}^{j} \neq 0 $$
for some $s \in \mathbb{N} \setminus \{ 0 \}$, then we have,
respectively
$$ \nu_{-} \neq 0 \ \ \mbox{and } \ \ v(\nu_{-} ) \in
   \Gamma_{\langle \lambda \rangle} - (\mathbb{N} \setminus \{ 0 \}) \cdot
   v(\mu) $$
or
$$ \nu_{+} \neq 0 \ \ \mbox{and } \ \ v(\nu_{+} ) \in
   \Gamma_{\langle \lambda \rangle} + (\mathbb{N} \setminus \{ 0 \}) \cdot
   v(\mu). $$
\end{lemma}


\begin{proof}
Consider first the latter case. Clearly,
$$ \nu_{+} =
   \sum_{l=1}^{\infty} \left( \sum_{j=0}^{\infty} \ a_{j+l,j}(\tilde{\lambda})
   \cdot \varepsilon_{1}^{j} \right) \cdot \mu^{l}, $$
and the values of (the valuation $v$ taken on) the $l$-th summands
of the above series are pairwise distinct, unless they are
infinity. Consequently, $v(\nu_{+})$ is the minimum of the values
of those summands, because some summands are non-vanishing (for
instance, the $s$-th one). Hence
$$ v(\nu_{+}) < \infty \ \ \mbox{and } \ \ v(\nu_{+} ) \in
   \Gamma_{\langle \lambda \rangle} + (\mathbb{N} \setminus \{ 0 \}) \cdot
   v(\mu), $$
as desired.

\vspace{1ex}

In the former case, take $n$ large enough so that
$$ v(\mu^{n}) >  v \left( \sum_{i=0}^{\infty} \ a_{i,i+s}(\tilde{\lambda}) \cdot
   \mu^{i} \cdot \left( \frac{\varepsilon_{1}}{\mu} \right)^{i+s}
   \right), $$
and write down $\nu_{-}$ as follows:
$$  \nu_{-} = \ \sum_{l=1}^{s} \left( \sum_{i=0}^{\infty} \ a_{i,i+l}(\tilde{\lambda}) \cdot
   \mu^{i} \cdot \left( \frac{\varepsilon_{1}}{\mu} \right)^{i+l}
   \right) + $$
$$   + \sum_{i=0}^{n-1} \left( \frac{1}{i!} \cdot
     \frac{\partial^{i}f}{\partial
     u^{i}}(0,\frac{\varepsilon_{1}}{\mu},\tilde{\lambda})
     - \sum_{j\leq i+s} \ a_{ij}(\tilde{\lambda}) \cdot
     \left( \frac{\varepsilon_{1}}{\mu} \right)^{j} \right)
     \cdot \mu^{i} \ + $$
$$   + \sum_{i=n}^{\infty} \left( \frac{1}{i!} \cdot
     \frac{\partial^{i}f}{\partial
     u^{i}}(0,\frac{\varepsilon_{1}}{\mu},\tilde{\lambda})
     - \sum_{j\leq i+s} \ a_{ij}(\tilde{\lambda}) \cdot
     \left( \frac{\varepsilon_{1}}{\mu} \right)^{j} \right)
     \cdot \mu^{i}. $$
Observe again that the $\mathcal{Q}$-analytic functions which
occur in the $i$-th summands of the second series above are of the
form
 $$ \left( \frac{1}{i!} \cdot
     \frac{\partial^{i}f}{\partial
     u^{i}}(0,v,\tilde{x})
     - \sum_{j\leq i+s} \ a_{ij}(\tilde{x}) \cdot
     v^{j} \right) \cdot u^{i} = h_{i}(v,\tilde{x}) \cdot v^{i+s+1} \cdot u^{i} $$
 for some functions $h_{i}$ $\mathcal{Q}$-analytic at zero, $i=0,\ldots,n$. Hence
 $$ \sum_{i=0}^{n-1} \left( \frac{1}{i!} \cdot
     \frac{\partial^{i}f}{\partial
     u^{i}}(0,\frac{\varepsilon_{1}}{\mu},\tilde{\lambda})
     - \sum_{j\leq i+s} \ a_{ij}(\tilde{\lambda}) \cdot
     \left( \frac{\varepsilon_{1}}{\mu} \right)^{j} \right)
     \cdot \mu^{i} = $$
 $$  = \left( \frac{\varepsilon_{1}}{\mu} \right)^{s+1}
   \cdot \sum_{i=0}^{n} \ h_{i}(\frac{\varepsilon_{1}}{\mu},\tilde{\lambda})
   \cdot \varepsilon_{1}^{i}. $$
By {\em op.~cit.\/}, Corollary~2.11,  we have
$$ v(h_{i}(\frac{\varepsilon_{1}}{\mu},\tilde{\lambda})) \in
   \Gamma_{\langle \lambda \rangle} \oplus \mathbb{N} \cdot
   v(\frac{\varepsilon_{1}}{\mu}). $$
Since
$$ \sum_{i=0}^{\infty} \ a_{i,i+l}(\tilde{\lambda}) \cdot
   \mu^{i} \cdot \left( \frac{\varepsilon_{1}}{\mu} \right)^{i+l}
   = \left( \frac{\varepsilon_{1}}{\mu}
   \right)^{l} \cdot \sum_{i=0}^{\infty} \ a_{i,i+l}(\tilde{\lambda}) \cdot
   \varepsilon_{1}^{i} \cdot , \ \ \ l=1,\ldots,s, $$
the values
$$ v \left( \sum_{i=0}^{\infty} \ a_{i,i+l}(\tilde{\lambda}) \cdot
   \mu^{i} \cdot \left( \frac{\varepsilon_{1}}{\mu} \right)^{i+l}
   \right) , \ \ \ l=1,\ldots,s, $$
and
$$   v \left( \sum_{i=0}^{n-1} \left( \frac{1}{i!} \cdot
     \frac{\partial^{i}f}{\partial
     u^{i}}(0,\frac{\varepsilon_{1}}{\mu},\tilde{\lambda})
     - \sum_{j\leq i+s} \ a_{ij}(\tilde{\lambda}) \cdot
     \left( \frac{\varepsilon_{1}}{\mu} \right)^{j} \right)
     \cdot \mu^{i} \right) $$
are pairwise distinct, unless they are infinity. Consequently,
$v(\nu_{-})$ is the minimum of the above $(s+1)$ values. Hence
$$ v(\nu_{-}) < \infty \ \ \mbox{and } \ \ v(\nu_{-} ) \in
   \Gamma_{\langle \lambda \rangle} - (\mathbb{N} \setminus \{ 0 \}) \cdot
   v(\mu), $$
which completes the proof of Lemma~\ref{proof-lem-2}.
\end{proof}


Now, take $n \in \mathbb{N}$ large enough so that
$$ n \cdot v(\varepsilon_{1}) > \min \, \{ \, v(\nu_{-}),v(\nu_{+}) \, \}. $$
Then
$$ v \left( \nu - \sum_{i=0}^{n-1} \, a_{i,i}(\tilde{\lambda}) \cdot
   \varepsilon_{1}^{i} \right) = \min \, \{ v(\nu_{-}), v(\nu_{+}) \}
   \not \in \Gamma_{\langle \lambda \rangle}. $$
This means that $\nu$ is active over the infinitesimals $\lambda$,
concluding the proof of the theorem on an active infinitesimal
$\mathbf{(I_{m})}$ in Case~A.

\vspace{2ex}

{\bf CASE B.} Let us rename the coordinates in $\mathbb{R}^{m}$ in
the following way: the coordinates $x=(x_{1},\dots,x_{q})$
correspond to the infinitesimals $\delta$ and the coordinates
$y=(y_{1},\dots,y_{p})$ correspond to the infinitesimals
$\varepsilon$. In the new variables, the function $f$ can be
written down as $f(u,v,\tilde{x},y)$, $\tilde{x} =
(x_{2},\ldots,x_{q})$. We first establish the following

\vspace{2ex}

{\bf Reduction Step.}
\begin{em}
We can assume that the infinitesimal $\nu$ is of the form $\nu =
f(\mu,\delta_{1}/\mu,\tilde{\delta},\varepsilon)$, where $f$ is a
function $\mathcal{Q}$-analytic at zero, and
\end{em}
$$ v(\varepsilon) > G_{2} \oplus \ldots \oplus G_{r}
   \ \ \mbox{ and } \ \ \Gamma_{\langle \delta,\mu \rangle} <
   G_{1}^{+}. $$

\vspace{1ex}

{\em Proof.} First, we recursively attach the old infinitesimals
$\delta_{2},\ldots,\delta_{q}$, after performing suitable special
modifications, either to the new infinitesimals $\delta'$ or to
the new infinitesimals $\varepsilon'$, so as to fulfil the
conditions
\begin{equation}\label{con-1}
 v(\varepsilon') > G_{2} \oplus \ldots \oplus G_{r}
   \ \ \mbox{ and } \ \ \Gamma_{\langle \delta' \rangle} <
   G_{1}^{+}.
\end{equation}

At the beginning, take as new infinitesimals $\delta'$ those
infinitesimals from $\delta_{2},\ldots,\delta_{q}$, which lie in
the main part of the regular sequence
$\mu,\lambda_{1},\ldots,\lambda_{k}$. Having constructed a
sequence $\delta_{2}',\ldots,\delta_{i}'$, take an infinitesimal
$\delta_{j}$ from among $\delta_{2},\ldots,\delta_{q}$, which has
not yet been considered in the process. If
$$ \Gamma_{\langle \delta_{2}',\ldots,\delta_{i}',\delta_{j} \rangle} < G_{1}^{+}, $$
attach $\delta_{j} =: \delta_{i+1}'$ to the new infinitesimals
$\delta'$. Otherwise $\delta_{j}$ is active over
$\delta_{2}',\ldots,\delta_{i}'$. By the valuation property
$\mathbf{(III_{m})}$, which holds by the induction hypothesis,
there is an $\mathcal{L}$-term
$\tau(\delta_{2}',\ldots,\delta_{i}')$ such that
$$ v(\delta_{j} -  \tau(\delta_{2}',\ldots,\delta_{i}')) > G_{2}
   \oplus \ldots \oplus G_{r}. $$
Via desingularization of $\mathcal{L}$-terms
(Corollary~\ref{rect-cor-2} and \cite[Corollary~2.6]{Now2}), we
can assume, after a suitable change of the infinitesimals
$\delta_{2}',\ldots,\delta_{i}'$ by special modification, that
$$ \tau(\delta_{2}',\ldots,\delta_{i}') =
   \varphi(\delta_{2}',\ldots,\delta_{i}'), $$
where $\varphi$ is a function $\mathcal{Q}$-analytic at zero. Then
we attach the infinitesimal
$$ \omega:= \delta_{j} - \varphi(\delta_{2}',\ldots,\delta_{i}') $$
to the new infinitesimals $\varepsilon$. By substitution $\omega +
\varphi(\delta_{2}',\ldots,\delta_{i}')$ for $\delta_{j}$, we are
done. We continue this process until all infinitesimals
$\delta_{2},\ldots,\delta_{q}$ have been considered. In this
manner, we get infinitesimals $\delta_{2}',\ldots,\delta_{t}'$ and
$\epsilon_{1}',\ldots,\epsilon_{i}'$ that fulfil
conditions~\ref{con-1}.

\vspace{1ex}

Next, consider the infinitesimal $\delta_{1}$. If
$$ \Gamma_{\langle \delta_{2}',\ldots,\delta_{t}',\delta_{1} \rangle} < G_{1}^{+}, $$
we are done by putting $\delta_{1}' := \delta_{1}$. Otherwise
$\delta_{1}$ is active over $\delta_{2}',\ldots,\delta_{t}'$. As
before, by the valuation property $\mathbf{(III_{m})}$ and via
desingularization of $\mathcal{L}$-terms, we can assume that
$$ v(\omega)> G_{2} \oplus \ldots \oplus G_{r} \ \ \mbox{ with } \ \
   \omega:= \delta_{1} - \varphi(\delta_{2}',\ldots,\delta_{t}'), $$
where $\varphi$ is a function $\mathcal{Q}$-analytic at zero. Let
$\delta'=(\delta_{2}',\ldots,\delta_{t}')$. Then, similarly to
{\em op.cit.}, Section~4, we can replace the function $f$ by some
other $\mathcal{Q}$-analytic functions as follows:
$$ \nu = f(\mu,(\varphi(\delta') +
    \omega)/\mu,\delta',\varepsilon') =
    f_{1}(\mu,\varphi(\delta')/\mu, \omega/\mu,\delta',\varepsilon')
    = $$
$$  = f_{2}(\mu,\varphi(\delta')/\mu,
    \omega/\varphi(\delta'),\delta',\varepsilon'). $$
We can thus attach the infinitesimal $\varepsilon_{s+1}' :=
\omega/\varphi(\delta')$ to the infinitesimals $\varepsilon'$, and
then
$$ \nu = f_{2}(\mu,\varphi(\delta')/\mu,\varepsilon_{s+1}',\delta',\varepsilon'). $$

For simplicity, we drop the sign of apostrophe over the name of
infinitesimals. By transforming the function $\varphi$ to a normal
crossing, we may assume that
$$  \nu = f_{3}(\mu,\delta^{\alpha}/\mu,\delta,\varepsilon) $$
for some $\alpha \in \mathbb{N}^{q}$. Replacing the infinitesimals
$\mu$ and $\delta$ by their suitable roots, we may assume that
$\delta^{\alpha}= \delta_{1} \cdot \ldots \cdot \delta_{k}$ for
some $k \leq q$. Now, we can successively lower the number $k$ of
these factors as follows. Since $v(\mu) \not \in \Gamma_{\langle
\delta \rangle}$, exactly one of the two fractions $\delta_{1}
\cdot \ldots \cdot \delta_{k-1}/\mu$ or $\mu/\delta_{1} \cdot
\ldots \cdot \delta_{k-1}$ is an infinitesimal. In the former
case, we get
$$  \nu = f_{3}(\mu,\delta_{1} \cdot \ldots \cdot \delta_{k}/\mu,\delta,\varepsilon) =
    f_{4}(\mu,\delta_{1} \cdot \ldots \cdot \delta_{k-1}/\mu,\delta,\varepsilon); $$
and in the latter
$$  \nu = f_{3}(\mu,\delta_{1} \cdot \ldots \cdot \delta_{k}/\mu,\delta,\varepsilon) =
    f_{4}(\mu,\mu/\delta_{1} \cdot \ldots \cdot \delta_{k-1},
    \delta_{1} \cdot \ldots \cdot \delta_{k}/\mu,\delta,\varepsilon)
    = $$
$$  =  f_{5}(\mu/\delta_{1} \cdot \ldots \cdot \delta_{k-1},
    \delta_{1} \cdot \ldots \cdot
    \delta_{k}/\mu,\delta,\varepsilon), $$
and thus, replacing $\mu$ by $\mu' := \mu/\delta_{1} \cdot \ldots
\cdot \delta_{k-1}$, we get
$$ \nu = f_{5}(\mu',\delta_{k}/\mu',\delta,\varepsilon). $$
Again, we drop the sign of apostrophe. Eventually, we can assume
that
$$ \nu = f_{6}(\mu,\delta_{1}/\mu,\delta_{1},\tilde{\delta},\varepsilon) =
    f_{7}(\mu,\delta_{1}/\mu,\tilde{\delta},\varepsilon), $$
which is the desired result. \hspace{\fill} $\Box$

\vspace{2ex}

Summing up, the above construction yields a new regular sequence
of infinitesimals $\delta,\varepsilon$ which satisfy
conditions~\ref{con-1}. To finish the reduction step, we must
still show that $\Gamma_{\langle \delta,\mu \rangle} < G_{1}^{+}$.
But this follows immediately from the valuation property
$\mathbf{(III_{m})}$, which is at our disposal by the induction
hypothesis, applied to the infinitesimals $\delta$ and $\mu$.

\vspace{2ex}

For the rest of the proof of $\mathbf{(I_{m})}$, we shall keep the
conditions established in the reduction step. In the completion of
the valued field $\langle \lambda,\mu \rangle$ with respect to the
standard valuation $v$, we can present the infinitesimal $\nu$ in
the following form:
$$ \nu = f(\mu,\delta_{1}/\mu,\tilde{\delta},\varepsilon) =
   \sum_{\alpha \in \mathbb{N}^{p}} \; \varepsilon^{\alpha} \cdot
   f_{\alpha}(\mu,\delta_{1}/\mu,\tilde{\delta}), $$
where
$$ f_{\alpha}(u,v,\tilde{x}) := \frac{1}{\alpha!} \cdot
   \frac{\partial^{|\alpha|}f}{\partial y^{\alpha}}
   (u,v,\tilde{x},0), \ \ \ \alpha \in \mathbb{N}^{p}. $$

We need an elementary fact about the standard valuation $v$.


\begin{lemma}\label{proof-lem-3}
Consider a finite number of elements $g_{i},h_{i} \in \langle
\lambda,\mu \rangle$, $i=1,\ldots,n$, such that
$$ v(h_{1}),\ldots,v(h_{n}) > G_{2} \oplus \ldots \oplus G_{r} $$
and
$$ v \left (\sum_{i=1}^{n} \: c_{i} \cdot g_{i} \right) < G_{1}^{+}, \ \ \
   c_{i} \in \mathbb{R}, \ i=1,\ldots,n, $$
for all real linear combinations of the elements
$g_{1},\ldots,g_{n}$. Then there exist $n$ real linear
combinations
$$ G_{j} = \sum_{i=1}^{n} \, c_{ji}g_{i}, \ \ H_{j}
   = \sum_{i=1}^{n} \, d_{ji}h_{i},
   \ \ \ c_{ji},d_{ji} \in \mathbb{R}, \ \ i,j = 1,\ldots,n, $$
such that
$$ \sum_{i=1}^{n} \, h_{i} \cdot g_{i} =
   \sum_{i=1}^{n} \, H_{i} \cdot G_{i},  $$
and that the valuations $v(H_{1}),\ldots,v(H_{n})$ are pairwise
distinct; then, of course, the valuations $v(H_{1} \cdot
G_{1}),\ldots,v(H_{n} \cdot G_{n})$ are pairwise distinct.
\end{lemma}


\begin{proof}
One can proceed with induction with respect to $n$. We show the
case $n=2$. The general case is similar, the details being left to
the reader. If $v(h_{1}) \neq v(h_{1})$, we are done. Otherwise
there are real numbers $d_{1}, d_{2} \neq 0$ such that
$$ v(d_{1}h_{1} - d_{2}h_{2}) > v(h_{1}) = v(h_{2}), $$
and then
$$ h_{1} \cdot g_{1} + h_{2} \cdot g_{2} =
   d_{1} h_{1} \cdot d_{1}^{-1}g_{1} + d_{2} h_{2} \cdot d_{2}^{-1} g_{2}
   =  $$
$$ = d_{1} h_{1} \cdot (d_{1}^{-1}g_{1} + d_{2}^{-1} g_{2}) +
   (d_{2} h_{2} - d_{1} h_{1}) \cdot d_{2}^{-1} g_{2}. $$
Putting
$$ H_{1} := d_{1}h_{1}, \ H_{2} := d_{2}h_{2} - d_{1}h_{1}, \
   G_{1} := d_{1}^{-1}g_{1} + d_{2}^{-1}g_{2}, \ G_{2} :=
   d_{2}^{-1}g_{2}, $$
we get the required result.
\end{proof}


Applying Lemma~\ref{proof-lem-3} to the elements
$$ h_{\alpha} = \varepsilon^{\alpha}  \ \ \mbox{ and } \ \
   g_{\alpha} = f_{\alpha}(\mu,\delta_{1}/\mu,\tilde{\delta}),
   \ \ \ \alpha \in \mathbb{N}^{p}, $$
we shall recursively define an increasing sequence of positive
integers $(N_{k})$ and three sequences of infinitesimals
$(f_{k})$, $(\phi_{k})$ and $(\psi_{k})$, $k \in \mathbb{N}$.
Initially, let $\gamma \in G_{1}^{+}$, $N_{0}$ be any positive
integer and
$$ A_{0} := \{ \alpha \in \mathbb{N}^{p}: \ v(\varepsilon^{\alpha}) <
   N_{0} \gamma \, \} , \ \ n_{0} := \sharp \, A_{0}.  $$
With the notation of Lemma~\ref{proof-lem-3}, we get
$$ f_{0} := \sum_{\alpha \in A_{0}} \; \varepsilon^{\alpha} \cdot
   g_{\alpha} = \sum_{i=1}^{n_{0}}  H_{0,i} \cdot G_{0,i}; $$
put
$$ \psi_{0} := 0, \ \ \varphi_{0} := \sum \{ H_{0,i} \cdot G_{0,i}: \; v(H_{i}) < N_{0}
   \gamma \, \} \ \ \mbox{ and } \ \ \psi_{1} := f_{0} -
   \varphi_{0}. $$
Take $N_{1}$ so large that $v(H_{0,i}) < N_{1} \gamma$ for all
$i=1,\ldots,n_{0}$. Let
$$ A_{1} := \{ \alpha \in \mathbb{N}^{p}: \ N_{0}\gamma \leq v(\varepsilon^{\alpha}) <
   N_{1} \gamma \, \}  $$
and $n_{1}$ be the sum of $\sharp \, A_{1}$ and the number of the
summands of $\psi_{1}$. Again, with the notation of
Lemma~\ref{proof-lem-3}, we get
$$ f_{1} := \sum_{\alpha \in A_{1} \setminus A_{0}} \; \varepsilon^{\alpha} \cdot
   g_{\alpha} + \psi_{1} = \sum_{i=1}^{n_{1}}  H_{1,i} \cdot G_{1,i}; $$
clearly, $v(H_{1,i}) \geq N_{0} \gamma$ for $i=1,\ldots,n_{1}$;
put
$$ \varphi_{1} := \sum_{i} \{ H_{1,i} \cdot G_{1,i}: \; v(H_{1,i}) < N_{1}
   \gamma \, \} \ \ \mbox{ and } \ \ \psi_{2} := f_{1} -
   \varphi_{1}. $$
We continue this process recursively. By construction, each
$\varphi_{k}$ is a finite sum of the form
$$ \varphi_{k}:= \sum_{i} H_{k,i} \cdot G_{k,i}\, , \ \ \mbox{ where }
   \ N_{k-1}\gamma \leq v(H_{k,i}) < N_{k}\gamma \ \mbox{ for all }
   i, $$
and the values $v(H_{k,i})$ are pairwise distinct. It is easy to
check that
$$ \nu = \sum_{\alpha \in \mathbb{N}^{p}} \; \varepsilon^{\alpha} \cdot
   g_{\alpha} = \sum_{k=0}^{\infty} \: f_{k} =
   \sum_{k=0}^{\infty} \: \varphi_{k}. $$

We encounter two possibilities: either $G_{k,i} \in \langle \delta
\rangle$ for each $k \in \mathbb{N}$ and every $i$, or there is a
$k \in \mathbb{N}$ such that $G_{k,i} \not \in \langle \delta
\rangle$ for some $i$. We first show that the former one leads to
a contradiction. Indeed, for each $k \in \mathbb{N}$, $G_{k,i}$
are of the form
$$ G_{k,i} = G_{k,i}(\mu,\delta_{1}/\mu,\tilde{\delta}), \ \ \ i=1,\ldots,n_{k},$$
where $G_{k,i}(u,v,\tilde{x})$ are functions
$\mathcal{Q}$-analytic in a common neighbourhood of zero.

Since the infinitesimals $\mu, \delta$ are analytically
independent, the tantamount conditions
$$ G_{k,i}(\mu,\delta_{1}/\mu,\tilde{\delta}) \in
   \langle \delta \rangle \ \ \ \text{or} \ \ \
   G_{k,i}(\mu,\delta_{1}/\mu,\tilde{\delta}) = \tau(\delta) $$
for an $\mathcal{L}$-term $\tau(x)$, hold iff the
$\mathcal{L}$-term
$$ G_{k,i}(u,x_{1}/u,\tilde{x}) - \tau(x) $$
vanishes on an open special cube containing the infinitesimals
$\mu, \delta$. Then the partial derivative
$$  \partial/\partial u \, G_{k,i}(u,x_{1}/u,\tilde{x})  $$
vanishes on that special cube and, in particular,
$$  \partial/\partial u \, G_{k,i}(\mu,\delta_{1}/\mu,\tilde{\delta}) =0. $$

Therefore, the former possibility implies that the partial
derivative
$$ \partial/\partial u \, f(u,x_{1}/u,\tilde{x},y) $$
would vanish for $u=\mu$, $x=\delta$ and $y=\varepsilon$. Again,
this partial derivative would vanish on a special cube containing
the infinitesimals $\mu$, $\delta$ and $\varepsilon$, and thus it
would vanish identically, by the identity principle for
quasianalytic functions. As before, when considering Case~A, we
deduce that the function
$$ f(u,x_{1}/u,\tilde{x},y) $$
would coincide with a function $g(x,y)$ $\mathcal{Q}$-analytic at
zero. Hence
$$ \nu = g(\delta,\varepsilon) \in \langle \delta,\varepsilon \rangle =
    \langle \lambda \rangle, $$
contrary to the assumption of $\mathbf{(I_{m})}$.

\vspace{1ex}

In this fashion, we may assume that the latter possibility holds.
Obviously, we can write down the infinitesimal $\nu$ as below
$$ \nu = \sum_{k=0}^{\infty} \: \varphi_{k} = \sum_{j=0}^{\infty} \:
   H_{j} \cdot G_{j}, $$
where the values $v(H_{j})$, $j \in \mathbb{N}$, are pairwise
distinct and for any $\gamma \in G_{1}$ there are only finitely
many $j$ for which $v(H_{j}) < \gamma$. Under the circumstances,
the set
$$ J := \{ j \in \mathbb{N}: \; G_{j} \not \in \langle \delta \rangle \} \neq
   \emptyset $$
is non-empty. There is, of course, a unique $j_{0} \in J$ such
that
$$ v(H_{j_{0}}) = \min \, \{ v(H_{j}): \, j \in J \}. $$
Then
$$ I := \{ j \in \mathbb{N}: \: v(H_{i}) < v(H_{i_{0}}) \} $$
is a finite subset of $\mathbb{N}$.

Now, it follows from the induction hypothesis $\mathbf{(I_{m-1})}$
that there is an element $\tau(\delta) \in \langle \delta \rangle$
such that
$$ v(G_{j_{0}} - \tau(\delta)) \not \in \Gamma_{\langle \delta \rangle}. $$
Then
$$ \Lambda := H_{j_{0}} \cdot \tau(\delta) + \sum_{j \in I}
   \; H_{j} \cdot G_{j} \in \langle \delta,\varepsilon \rangle =
   \langle \lambda \rangle, $$
because $H_{j} \in \langle \varepsilon \rangle$ and $G_{j} \in
\langle \delta \rangle$ for all $j \in I$. It is not difficult to
check that
$$ v(\nu - \Lambda) = v(H_{j_{0}}) + v(G_{j_{0}} - \tau(\delta))
   \not \in \Gamma_{\langle \delta,\varepsilon \rangle} =
   \Gamma_{\langle \lambda \rangle}. $$
This means that $\nu$ is active over the infinitesimals $\lambda$,
which completes the proof of the theorem on an active
infinitesimal $\mathbf{(I_{m})}$.

\section{Quantifier elimination and description of definable
functions by terms}

In this section we are going to develop an approach to quantifier
elimination and description of definable functions by terms in the
language augmented by the names of rational powers, which is much
shorter and more natural than the one in~\cite{Now2}. Observe
first that one can introduce a well-defined notion of the
dimension of sets defined piecewise by $\mathcal{L}$-terms.
Indeed, every such set $E$ is a finite union of special cubes
$S_{i}$ ({\em op.~cit.\/}, Theorem~2.1 and Corollary~2.3), and one
can put
$$ \dim\, E := \max\, \dim\, S_{i}. $$
It is easy to check that the dimension of a set $E$ does not
depend on the decomposition into special cubes $S_{i}$.

\vspace{1ex}

By virtue of the exchange property for $\mathcal{L}$-terms~from
Section~3, $\mathbf{(IV_{m})}$, $m \in \mathbb{N}$,  ({\em
op.~cit.}, Corollary~4.9), we have at our disposal a well-founded
concept of rank for the substructures of a given model
$\mathcal{R}$ of the universal diagram $T$. This allows us to
establish the following result, which is a generalization of {\em
op.~cit.\/}, Proposition~5.4.

\begin{proposition}\label{proof-prop-1}
Consider a map $f: \mathbb{R}^{d} \longrightarrow \mathbb{R}^{m}$
given piecewise by $\mathcal{L}$-terms and such that for every
special cube $S \subset \mathbb{R}^{m}$ or, equivalently, for
every subset $E$ of $\mathbb{R}^{m}$ given piecewise by
$\mathcal{L}$-terms, we have
$$ \dim\, f^{-1}(S) \leq \dim\, S \mbox{ or } \
   \dim\, f^{-1}(E) \leq \dim\, E. $$
Then $f$ admits a section given piecewise by $\mathcal{L}$-terms,
i.e.\ there is a function $\xi: \mathbb{R}^{m} \longrightarrow
\mathbb{R}^{d}$ given piecewise by $\mathcal{L}$-terms such that
$f(\xi(y))=y$ for every point $y \in \mathbb{R}^{m}$.
\end{proposition}

\begin{proof}
We may, of course, assume that $f: (0,1)^{d} \longrightarrow
(0,1)^{m}$. We first show that there exists a family
$(t_{\iota}(y))_{\iota \in I}$ of $\mathcal{L}$-terms,
$t_{\iota}(y)=(t_{\iota,1}(y),\ldots,t_{\iota,d}(y))$, such that
the infinite disjunction
$$ \bigvee_{\iota \in I} \; \left[ \left( b=f(a) \wedge
   a \in (0,1)^{d} \right) \ \Longrightarrow \
   b = f(t_{\iota}(b)) \right] $$
holds for any tuples $a \in (0,1)^{d}$ and $b \in (0,1)^{m}$ in an
arbitrary model $\mathcal{R}$ of the theory $T$. So take any
elements $a \in (0,1)^{d}$ and $b \in (0,1)^{m}$ for which
$b=\varphi(a)$. We may, of course, confine our analysis to the
case where $a=\lambda$ and $b=\mu$ are infinitesimals. Let $k :=
\mbox{rk}\, \langle \lambda \rangle$; then the infinitesimals
$\lambda$ lie on a special cube of dimension $k$, but lie on no
special cube of dimension $<k$. Obviously,
$$ \langle \mu \rangle \subset \langle \lambda \rangle \ \ \mbox{
   and } \ \ \mbox{rk}\, \langle \mu \rangle \leq
   \mbox{rk}\, \langle \lambda \rangle . $$
Were $\mbox{rk}\, \langle \mu \rangle < k = \mbox{rk}\, \langle
\lambda \rangle$, then the infinitesimals $\mu$ would lie on a
special cube $S$ of dimension $<k$, and thus it follows from the
assumption that the infinitesimals $\lambda$ would lie on the set
$f^{-1} (S)$ of dimension $<k$, which is impossible. Consequently,
$$ \mbox{rk}\, \langle \mu \rangle =
   \mbox{rk}\, \langle \lambda \rangle  \ \ \mbox{ and } \ \
   \langle \mu \rangle = \langle \lambda \rangle ; $$
the last equality follows from the suitable property of rank
operation ({\em op.~cit.}, Section~5), which is formulated in
Section~3, assertion $\mathbf{(V_{m})}$, $m \in \mathbb{N}$.
Therefore the infinitesimals $\lambda$ can be expressed by
$\mathcal{L}$-terms taken on the infinitesimals $\mu$, and the
assertion follows.

\vspace{1ex}

Now, through model-theoretic compactness, one can find a finite
set $\iota_{1},\ldots,\iota_{n} \in I$ of indices for which the
finite disjunction
$$ \bigvee_{k=1}^{n} \left[ \left( b=f(a) \wedge a \in (0,1)^{d} \right)
   \ \Longrightarrow \ b = f(t_{\iota_{k}}(b)) \right] $$
holds for any tuples $a$ and $b$ in an arbitrary model
$\mathcal{R}$ of the theory $T$. In particular, this finite
disjunction is satisfied in the standard model
$\mathbb{R}_{\mathcal{Q}}$, concluding the proof of the
proposition.
\end{proof}

\begin{remark}\label{proof-rem-2} The assumption of Proposition~\ref{proof-prop-1}
is satisfied by every function $f$ given piecewise by
$\mathcal{L}$-terms which is an immersion. More generally,
consider a decomposition of $\mathbb{R}^{d}$ into finitely many
leaves (i.e.\ $\mathcal{C}^{1}$ submanifolds) $F_{j}$ given
piecewise by $\mathcal{L}$-terms, and a function $f$ given
piecewise by $\mathcal{L}$-terms whose restriction to each $F_{j}$
is an immersion. Then $f$ satisfies that assumption too.
\end{remark}

\begin{corollary}\label{QE-cor-1}
Under the assumptions of Proposition~\ref{proof-prop-1}, the image
$f(\mathbb{R}^{d})$ is given piecewise by $\mathcal{L}$-terms.
\end{corollary}

\begin{proof}
Indeed, suppose the section $\xi$ is given by a finite number of
$\mathcal{L}$-terms
$$ \tau_{i}(y)=(\tau_{i1}(y),\ldots,\tau_{id}(y)), \ \ \
    i=1,\ldots,s, \ \ y=(y_{1},\ldots,y_{m}), $$
i.e. \ $x=\xi(y)$ \ iff  \ $\bigvee_{i=1}^{s} \ x=\tau_{i}(y)$.
Then
$$ y \in f(\mathbb{R}^{d})  \ \ \Leftrightarrow \ \ \bigvee_{i=1}^{s} \
   f(\tau_{i}(y))=y, $$
and thus the image $f(\mathbb{R}^{d})$ is given by
$\mathcal{L}$-terms $f(\tau_{i}(y))-y$, $i=1,\ldots,s$, as
desired.
\end{proof}

By an immersion cube $C \subset \mathbb{R}^{m}$ we mean the image
$\varphi((0,1)^{d}$ where $\varphi$ is a $\mathcal{Q}$-map in a
neighbourhood of the compact cube $[0,1]^{d}$, whose restriction
to $(0,1)^{d}$ is an immersion. As demonstrated in our
paper~\cite{Now1}, the theorem on decomposition into special cubes
along with the technique of fiber cutting make it possible to
decompose every bounded $\mathcal{Q}$-subanalytic set into
finitely many immersion cubes ({\em op.~cit.\/}, Corollary~1).
This, in turn, and Corollary~\ref{QE-cor-1}, immediately yield
quantifier elimination for the expansion
$\mathbb{R}_{\mathcal{Q}}$ of the real field with restricted
quasianalytic functions in the language $\mathcal{L}$ augmented by
the names of rational powers (cf.~\cite{Now2}, Theorem~5.8):

\begin{theorem}\label{QE-th} (Quantifier Elimination)
Every set definable in the structure $\mathbb{R}_{\mathcal{Q}}$ is
given piecewise by a finite number of $\mathcal{L}$-terms.
\end{theorem}

A fortiori, the structure $\mathbb{R}_{\mathcal{Q}}$ is model
complete. Let us mention that one can apply decomposition into
immersion cubes (cf.~\cite{Now1}) to prove that
$\mathbb{R}_{\mathcal{Q}}$ is a polynomially bounded, o-minimal
structure which admits $\mathcal{Q}$-analytic cell decomposition
(being established in~\cite{Rol-Spei-Wil}); the last result
requires an induction procedure (cf.~\cite{Now1}).

\vspace{1ex}

We now wish to turn to the problem of description of definable
function by $\mathcal{L}$-terms.

\begin{theorem}\label{DF-th}
Each definable function $f: \mathbb{R}^{m} \longrightarrow
\mathbb{R}$ is piecewise given by a finite number of
$\mathcal{L}$-terms.
\end{theorem}

\begin{proof}
Indeed, consider the graph $F \subset
\mathbb{R}^{m+1}=\mathbb{R}^{m}_{x} \times \mathbb{R}_{y}$ of the
function $f$ and denote by
$$ p: F \longrightarrow \mathbb{R}^{m}_{x} \ \ \mbox{and } \ \
   q: F \longrightarrow \mathbb{R}_{y} $$
the canonical projections. Via cell decomposition (of class
$\mathcal{C}^{1}$), the graph $F$ can be partitioned into finitely
many cells $C_{i}$ defined by $\mathcal{L}$-terms such that the
restriction of $p$ to each cell $C_{i}$ is an immersion and, in
fact, a diffeomorphism onto the image $p(C_{i})$. It follows from
Proposition~\ref{proof-prop-1} (cf.\ Remark~\ref{proof-rem-2})
that each inverse
$$ p^{-1}: p(C_{i}) \longrightarrow C_{i}, \ \ \ i=1,\ldots,s, $$
is given piecewise by finitely many $\mathcal{L}$-terms, and thus
so is the restriction of $f = q \circ p^{-1}$ to each set
$p(C_{i})$. This completes the proof.
\end{proof}

We immediately obtain the following three corollaries.

\begin{corollary}\label{QE-cor-2}
The structure $\mathbb{R}_{\mathcal{Q}}$ admits cell
decompositions defined piecewise by $\mathcal{L}$-terms, and hence
Skolem functions (of choice) given piecewise by
$\mathcal{L}$-terms.
\end{corollary}

\begin{corollary}\label{QE-cor-3}
The structure $\mathbb{R}_{\mathcal{Q}}$ is universally
axiomatizable. Hence its universal diagram $T$ admits quantifier
elimination (in the language $\mathcal{L}$) and
$\mathbb{R}_{\mathcal{Q}}$ can be embedded as a prime model into
each model of $T$. Consequently, in every model of $T$, each
definable function is defined piecewise by a finite number of
$\mathcal{L}$-terms.
\end{corollary}

\begin{corollary}\label{QE-cor-4} (Valuation Property for Definable Functions)
Consider a simple (with respect to definable closure) extension
$\mathcal{R}  \subset \mathcal{R} \langle a \rangle$ of
substructures in a fixed model of the theory $T$. Then we have the
following dichotomy:
$$ \mbox{ either } \ \ \dim \Gamma_{\mathcal{R} \langle a \rangle} =
   \dim \Gamma_{\mathcal{R}} \ \ \mbox{ or } \ \ \dim
   \Gamma_{\mathcal{R} \langle a \rangle} = \dim
   \Gamma_{\mathcal{R}}
   +1.
$$
In the latter case, one can find an element $r \in \mathcal{R}$
such that
$$ v(a-r) \not \in \Gamma_{\mathcal{R}}  \ \ \mbox{ and } \ \ \Gamma_{\mathcal{R}  \langle
   a \rangle} = \Gamma_{\mathcal{R}}  \oplus \mathbb{Q} \cdot v(a-r). $$
\end{corollary}

\begin{remark}
For the valuation property in the case of general, polynomially
bounded, o-minimal structures and its connection with the
preparation theorem in the sense of Parusi\'nski--Lion--Rolin,
see~\cite{Dries-S0,Dries-S} and also~\cite{Now0}.
\end{remark}

Similarly, we can immediately rephrase Theorem~\ref{rect-th} for
definable functions (cf.~\cite[Theorem~1]{Now3}):

\begin{theorem}\label{rect-def-th-1} (On Rectilinearization of Definable
Functions) \ If $f_{1},\ldots,f_{s}: \mathbb{R}^{m}
\longrightarrow \mathbb{R}$ are definable functions and $K$ is a
compact subset of $\mathbb{R}^{m}$, then there exists a finite
collection of modifications
$$ \varphi_{i}: [-1,1]^{m} \longrightarrow \mathbb{}R^{m}, \ \ \ \
i=1,\ldots,p, $$ such that

1) each $\varphi_{i}$ extends to a $\mathcal{Q}$-map in a
neighbourhood of the cube $[-1,1]^{m}$, which is a composite of
finitely many local blow-ups with smooth $\mathcal{Q}$-analytic
centers and local power substitutions;

2) the union of the images $\varphi_{i}((-1,1)^{m})$,
$i=1,\ldots,p$, is a neighbourhood of $K$.

3) for every bounded quadrant $Q_{j}$, $j=1,\ldots,3^{m}$, the
restriction to $Q_{j}$ of each function $f_{k} \circ \varphi_{i}$,
$k=1,\ldots,s$, $i=1,\ldots,p$, either vanishes or is a normal
crossing or a reciprocal normal crossing on $Q_{j}$.
\end{theorem}

Further, the following two results concerning rectilinearization
of definable functions from our paper~\cite{Now3} can be repeated
verbatim in the quasianalytic settings:

\begin{theorem}\label{rect-def-th-2}  (On Rectilinearization of a Definable Function)
Let $U \subset \mathbb{R}^{m}$ be a bounded open subset and $f: U
\longrightarrow \mathbb{R}$ be a definable function. Then there
exists a finite collection of modifications
$$ \varphi_{i}: [-1,1]^{m} \longrightarrow \mathbb{R}^{m}, \ \ \ \
i=1,\ldots,p, $$ such that

1) each $\varphi_{i}$ extends to a $\mathcal{Q}$-map in a
neighbourhood of the cube $[-1,1]^{m}$, which is a composite of
finitely many local blow-ups with smooth $\mathcal{Q}$-analytic
centers and local power substitutions;

2) each set $\varphi_{i}^{-1}(U)$ is a finite union of bounded
quadrants in $\mathbb{R}^{m}$;

3) each set $\varphi_{i}^{-1}(\partial U)$ is a finite union of
 bounded closed quadrants in $\mathbb{R}^{m}$ of dimension $m-1$;

4) $U$ is the union of the images $\varphi_{i}(\mbox{\rm Int}\,
(Q))$ with $Q$ ranging over the bounded quadrants contained in
$\varphi_{i}^{-1}(U)$, $i=1,\ldots,p$;

5) for every bounded quadrant $Q$,  the restriction to $Q$ of each
function $f \circ \varphi_{i}$ either vanishes or is a normal
crossing or a reciprocal normal crossing on $Q$, unless
$\varphi_{i}^{-1}(U) \cap Q = \emptyset$.
\end{theorem}

\begin{remark} One can formulate Theorem~\ref{rect-def-th-2}
for several definable functions $f_{1},\ldots,f_{s}$.
\end{remark}

It follows from points 1) and 2) that every bounded quadrant of
dimension $< m$ contained in $\varphi_{i}^{-1}(U)$ is adjacent to
a bounded quadrant of dimension $m$ (a bounded orthant) contained
in $\varphi_{i}^{-1}(U)$. Hence
$$ \varphi_{i}^{-1}(\overline{U}) = \overline{\varphi_{i}^{-1}(U)}, $$
and therefore point 4) implies that $\overline{U}$ is the union of
the images $\varphi_{i}(\overline{Q})$ of the closures of those
bounded quadrants of dimension $m$ (bounded orthants) $Q$ for
which $\varphi_{i}(Q) \subset U$, $i=1,\ldots,p$.

\vspace{1ex}

For a bounded orthant $Q$ contained in $\varphi_{i}^{-1}(U)$,
denote by $\mbox{dom}_{i}\,(Q)$ the union of $Q$ and all those
bounded quadrants that are adjacent to $Q$ and disjoint with
$\varphi_{i}^{-1}(\partial U)$; it is, of course, an open subset
of the closure $\overline{Q}$. Moreover, the open subset
$\varphi_{i}^{-1}(U)$ of the cube $[-1,1]^{m}$ coincides with the
union of $\mbox{dom}_{i}\,(Q)$, where $Q$ range over the bounded
orthants that are contained in $\varphi_{i}^{-1}(U)$, and with the
union of those bounded quadrants that are contained in
$\varphi_{i}^{-1}(U)$. Consequently, the union of the images
$\varphi_{i}(\mbox{\rm Int}\, (Q))$, where $Q$ range over the
bounded quadrants that are contained in $\varphi_{i}^{-1}(U)$,
coincides with the union of the images
$$ \varphi_{i}(\mbox{\rm dom}_{i}\, (Q) \cap (-1,1)^{m}), $$ where $Q$
range over the bounded orthants $Q$ that are contained in
$\varphi_{i}^{-1}(U)$. Thus we get

\vspace{2ex}

\begin{corollary}\label{rect-def-cor} (On Rectilinearization of a
Continuous Definable Function) Let $U$ be a bounded open subset in
$\mathbb{R}^{m}$ and $f: U \longrightarrow \mathbb{R}$ be a
continuous definable function. Then there exists a finite
collection of modifications
$$ \varphi_{i}: [-1,1]^{m} \longrightarrow \mathbb{R}^{m}, \ \ \ \
i=1,\ldots,p, $$ such that

1) each $\varphi_{i}$ extends to a $\mathcal{Q}$-map in a
neighbourhood of the cube $[-1,1]^{m}$, which is a composite of
finitely many local blow-ups with smooth $\mathcal{Q}$-analytic
centers and local power substitutions;

2) each set $\varphi_{i}^{-1}(U)$ is a finite union of bounded
quadrants in $\mathbb{R}^{m}$;

3) each set $\varphi_{i}^{-1}(\partial U)$ is a finite union of
bounded closed quadrants in $\mathbb{R}^{m}$ of dimension $m-1$;

4) $U$ is the union of the images $\varphi_{i}(\mbox{\rm
dom}_{i}\, (Q) \cap (-1,1)^{m})$ with $Q$ ranging over the bounded
orthants $Q$ contained in $\varphi_{i}^{-1}(U)$, $i=1,\ldots,p$;

5) for every bounded orthant $Q$, the restriction to $\mbox{\rm
dom}_{i}\, (Q)$ of each function $f \circ \varphi_{i}$ either
vanishes or is a normal crossing or a reciprocal normal crossing
on $Q$, unless $\varphi_{i}^{-1}(U) \cap Q = \emptyset$.
\end{corollary}

\begin{remark}\label{rem-proof-rem-3} Observe that in the foregoing
rectilinearization results, if the functions $f_{1},\ldots,f_{s}$
are given piecewise by terms in the language of restricted
$\mathcal{Q}$-analytic functions augmented merely by the
reciprocal function $1/x$, then one can require that the
modifications $\varphi_{i}$, $i=1,\ldots,p$, be composite of
finitely many local blow-ups with smooth $\mathcal{Q}$-analytic
centers.
\end{remark}

\section{Power substitution for Denjoy--Carleman classes}

Consider an increasing sequence $M=(M_{n})$ of real numbers with
$M_{0}=1$. Let $I$ be an interval (open or closed) contained in
$\mathbb{R}$. By $\mathcal{Q}(I,M)$ we denote the class of
functions on $I$ that satisfy the following growth condition for
their derivatives:
$$ \left| \partial^{|\alpha|} f
/\partial x^{\alpha} (x) \right| \leq C \cdot
   R^{|\alpha|} \cdot |\alpha|! \cdot M_{|\alpha|} \ \ \
   \mbox{ for all } \ \ x \in I, \ \alpha \in \mathbb{N}^{n}, $$
with some constants $C, R >0$.

\vspace{1ex}

The main purpose of this section is to prove the following

\begin{theorem}\label{substitution}
Let $p >1$ be an integer and $I$ the interval $[0,1]$ or $[-1,1]$
according as $p$ is even or odd. Consider power substitution $x =
\xi^{p}$, which is a bijection of $I$ onto itself. Let $f: I
\longrightarrow \mathbb{R}$ be a smooth function. If
$$ F(\xi) := (f \circ \varphi)(\xi) = f(\xi^{p}) \in
   \mathcal{Q}(I,M), $$
then \ \ $f(x) \in \mathcal{Q}(I,M^{(p)})$, \ where the sequence
$M^{(p)}$ is defined by putting $M^{(p)}_{n}:= M_{pn}$.

Equivalently, in terms of the corresponding sequences $M'$ (see
the Introduction), the function $f$ belongs to the class
determined by the sequence $M'^{(p)}_{n}:= 1/n^{(p-1)n} \cdot
M'_{pn}$, $n \in \mathbb{N}$.
\end{theorem}


\begin{remark}\label{rem1} Using a function constructed by Bang,
we shall show at the end of this section that, in the case where
$p=2$ and the sequence $M'$ is log-convex,
$\mathcal{Q}(I,M^{(2)})$ is the smallest Denjoy--Carleman class
containing all those functions $f(x)$.
\end{remark}


\begin{remark}\label{rem2} The case $p=2$ of Theorem \ref{substitution}
may be related to the following problem, which was investigated by
Mandelbrojt~\cite{M1,M2} and solved completely by Lalagu\"{e}
\cite[Chap.~III]{L}.

\vspace{1ex}

\begin{em}
Consider a smooth function $f(x)$ on the interval $[-1,1]$ and
suppose that $F(\xi):= f(\cos \xi)$ belongs to a class
$\mathcal{Q}(\mathbb{R},M)$. To which class on the interval
$[-1,1]$ does $f$ belong?
\end{em}
\end{remark}


\begin{proof}

Before establishing Theorem \ref{substitution}, we state two
lemmas below.


\begin{lemma}\label{lem1.}
Consider the Taylor expansions
$$ \sum_{i=1}^{\infty} \, \frac{x^{i}}{i} = - \log \, (1-x) \ \
   \mbox{ and } \ \ \left( \sum_{i=1}^{\infty} \, \frac{x^{i}}{i}
   \right) ^{k} \, = \sum_{n=1}^{\infty} \, c_{k,n} x^{n}. $$
Then we have the estimate
$$ c_{k,n} \leq (2e)^{n} \cdot \frac{k!}{n^{k}} \ \ \mbox{ for all } \ \
   k,n \in \mathbb{N}. $$
\end{lemma}


\begin{proof}
Indeed, it is easy to verify the estimate:
$$ |\log (1-z) | \leq \left| \log 2 + \frac{\pi}{6} \sqrt{-1} \right| \leq 1
   \ \ \mbox{ for all } \ \ z \in \mathbb{C}, \ |z| \leq 1/2. $$
By Cauchy's inequalities, we thus get $|c_{k,n}| \leq 2^{n}$.
Since $e^{n} > n^{k}/k!$ \ for all $n,k \in \mathbb{N}$, we have
$$ c_{k,n} \leq 2^{n} < 2^{n} \cdot e^{n} \cdot \frac{k!}{n^{k}} =
    (2e)^{n} \cdot \frac{k!}{n^{k}}, $$
as asserted.
\end{proof}

\vspace{1ex}

As an immediate consequence, we obtain

\vspace{1ex}

\begin{corollary}\label{cor}
$$  \sum_{i_{1}+\ldots+i_{k}=n} \, \frac{1}{i_{1}} \cdot \ldots
   \cdot \frac{1}{i_{k}} = c_{k,n} \leq (2e)^{n} \cdot \frac{k!}{n^{k}}
   \ \ \mbox{ for all } \ \ k,n \in \mathbb{N}. $$
\end{corollary}

\vspace{1ex}

\begin{lemma}\label{lem2}
Let $p,k \in \mathbb{N}$ with $p > 1$, $k \geq 1$, and
$$ \alpha_{k}(X,x) := \frac{1}{k!} \left( X^{1/p} - x^{1/p}
   \right)^{k}
   \ \ \mbox{ for } \ \ X,x > 0, $$
where $\alpha_{k}$ is regarded as a function in one variable $X$
and parameter $x$. Then we have the estimate
$$ \left| \alpha_{k}^{(n)}(x,x) \right| \leq (2e)^{n} \cdot n^{n-k}
   \cdot x^{- \frac{pn-k}{p}} \ \ \mbox{ for all } \ \
   n,k \in \mathbb{N}, \ \ x > 0. $$
\end{lemma}

\begin{proof}
Consider first the case $k=1$, \ $\alpha_{1}(X,x)= X^{1/p} -
x^{1/p}$. Then
$$ \alpha_{1}^{(n)}(x,x) = \pm \frac{(p-1)(2p-1) \cdot \ldots
   \cdot ((n-1)p -1)}{p^{n}} \cdot x^{- \frac{pn-1}{p}}, $$
whence
$$ \left| \alpha_{1}^{(n)}(x,x) \right| \leq (n-1)! \cdot x^{- \frac{pn-1}{p}}
   \leq n^{n-1} \cdot x^{- \frac{pn-1}{p}}, $$
as asserted. The Taylor expansion of $\alpha_{1}(X,x)$ at $X=x$ is
$$ \alpha_{1}(X,x) = \sum_{i=1}^{\infty} \, a_{i} \cdot x^{- \frac{pi-1}{p}}
   (X-x)^{i}, $$
where
$$ a_{i} := \pm \frac{1}{i!} \cdot \frac{(p-1)(2p-1) \cdot \ldots
   \cdot ((i-1)p -1)}{p^{i}}; $$
obviously, $|a_{i}| \leq (i-1)!/i! =1/i$. We thus get
$$ \alpha_{k}(X,x) = \frac{1}{k!} \left( \sum_{i=1}^{\infty} \, a_{i} \cdot x^{- \frac{pi-1}{p}}
   (X-x)^{i} \right)^{k} = \sum_{j=1}^{\infty} \, b_{j} \cdot x^{- \frac{pj-k}{p}}
   (X-x)^{j},  $$
where
$$ b_{j} := \frac{1}{k!} \sum_{i_{1}+\ldots+i_{k}=j} \, a_{i_{1}}
   \cdot \ldots \cdot a_{i_{k}}. $$
Then
$$ |b_{n} | \leq \frac{1}{k!} \sum_{i_{1}+\ldots+i_{k}=n} \,
   |a_{i_{1}}| \cdot \ldots \cdot |a_{i_{k}}| \leq \frac{1}{k!}
   \sum_{i_{1}+\ldots+i_{k}=n} \, \frac{1}{i_{1}} \cdot \ldots
   \cdot \frac{1}{i_{k}} = $$
$$  = \frac{c_{k,n}}{k!} \leq \frac{(2e)^{n}}{n^{k}}
   \ \ \mbox{ for all } \ \ n,k \in \mathbb{N}; $$
the last inequality follows from Corollary \ref{cor}. Hence
$$ \left| \alpha_{k}^{(n)}(x,x) \right| \leq n! \cdot | b_{n}|
   \cdot x^{- \frac{pn-k}{p}} \leq (2e)^{n} \cdot \frac{n!}{n^{k}}
   \cdot x^{- \frac{pn-k}{p}} \leq (2e)^{n} \cdot n^{n-k} \cdot x^{- \frac{pn-k}{p}} $$
for all $n,k \in \mathbb{N}$, $x > 0$, as asserted.
\end{proof}

\vspace{2ex}

Now we can readily pass to the proof of Theorem
\ref{substitution}. We shall work with estimates corresponding to
the sequence $M'$. So suppose that
$$ \left| F^{(n)}(\xi) \right| \leq A^{n} M'_{n} $$
for all $n \in \mathbb{N}$, $\xi \in I$ and some constant $A>0$.
We are going to estimate the growth of the $n$-th derivative
$f^{(n)}$. Fix $n \in \mathbb{N}$ and put:
$$ p_{n}(x) := T_{0}^{n}f(x) = \sum_{k=0}^{n-1} \, f^{k}(0)
   \frac{x^{k}}{k!}, \ \ \ r_{n}(x) := f(x) - p_{n}(x), $$
$$ P_{n}(\xi) := p_{n}(\xi^{p}) \ \ \mbox{ and } \ \ R_{n}(\xi) :=
   r_{n}(\xi^{p}). $$
Obviously,
$$ P_{n}^{(pn)} \equiv 0 \ \ \mbox{ and } \ \ R_{n}^{(pn)} \equiv
   F^{(pn)}. $$
From the Taylor formula, we therefore obtain the estimate
$$ \left| R_{n}^{(pn)}(\xi) \right| \leq A^{pn} M'_{pn} \cdot
   \frac{\xi^{pn-k}}{(pn-k)!} $$
for all $k < pn$, $\xi \in I$. We still need an elementary
inequality
\begin{equation}\label{element}
   \frac{1}{(pn-k)!} \leq \frac{e^{pn}}{n^{pn-k}} \ \ \mbox{ for
   all } \ \ k<pn.
\end{equation}
For, it suffices to show that
$$ \frac{e^{pn-k}}{(pn-k)^{pn-k}} \leq \frac{e^{pn}}{n^{pn-k}}. $$
When $pn-k \geq n$ or, equivalently, $k \leq (p-1)n$, the last
inequality is evident. Suppose thus that $pn-k<n$ or,
equivalently, $k>(p-1)n$. This inequality is, of course,
equivalent to
$$ \left( \frac{n}{pn-k} \right) ^{pn-k} \leq e^{k}, $$
which holds as shown below:
$$ \left( \frac{n}{pn-k} \right) ^{pn-k} =
   \left( 1 + \frac{k-(p-1)n}{pn-k} \right)^{pn-k} = $$
$$  = \left[ \left( 1 + \frac{k-(p-1)n}{pn-k}
   \right)^{\frac{pn-k}{k-(p-1)n}} \right] ^{k-(p-1)n}< e^{k}. $$

\vspace{1ex}

Now, the foregoing estimate along with inequality \eqref{element}
yield
$$ \left| R_{n}^{(pn)}(\xi) \right| \leq A^{pn} M'_{pn} \cdot \frac{e^{pn}}{n^{pn-k}}
    \cdot \xi^{pn-k}. $$
Applying the formula for the derivatives of a composite function,
we obtain
$$ r_{n}^{(n)}(x) = \sum_{k=1}^{n} \, R_{n}^{(k)}(\xi) \cdot
   \alpha_{k}^{(n)}(x,x). $$
Hence and by Lemma \ref{lem2}, we get
$$ \left| f^{(n)}(x) \right| = \left| r_{n}^{(n)}(x) \right| \leq
   \sum_{k=1}^{n} \, A^{pn} M'_{pn} \cdot \frac{e^{pn}}{n^{pn-k}}
   \cdot |\xi|^{pn-k} \cdot (2e)^{n} n^{n-k} \cdot |\xi|^{-(pn-k)} = $$
$$ = n \cdot (2e)^{n} \cdot (eA)^{pn} \cdot \frac{M'_{pn}}{n^{(p-1)n}}, $$
which completes the proof of Theorem \ref{substitution}.
\end{proof}

\vspace{2ex}

Finally, we show that, when the sequence $M'$ is log-convex,
$\mathcal{Q}(I,M^{(2)})$ is the smallest Denjoy--Carleman class
containing all smooth functions $f(x)$ on the interval $I =[0,1]$
such that $F(\xi) = f(\xi^{2}) \in \mathcal{Q}(I,M)$. We make use
of a classical function constructed by Bang~\cite{Ba}, applied in
his proof that the classes determined by log-convex sequences
contain functions with sufficiently large derivatives (the result
due to Cartan~\cite{C} and Mandelbrojt~\cite{C-M}; see also
\cite{Thil}, Section~1, Theorem~1).

\vspace{1ex}

The logarithmic convexity of the sequence $M'$ yields for every
$j,k \in \mathbb{N}$ the inequality
$$ \left( \frac{1}{m_{k}} \right)^{k-j} \leq \frac{M'_{j}}{M'_{k}} \ \
   \mbox{ where } \ \ m_{k} := \frac{M'_{k+1}}{M'_{k}}. $$
Consequently,
$$ F(\xi) := \sum_{k=0}^{\infty} \, \frac{M'_{k}}{(2m_{k})^{k}}
   \cos \, (2m_{k}\xi) $$
is an even smooth function on $\mathbb{R}$ such that
$$ F(\xi) \in \mathcal{Q}(\mathbb{R},M) \ \ \mbox{ and } \ \
   \left| F^{(2n)}(0) \right| \geq M'_{2n} $$
for all $n \in \mathbb{N}$. Therefore $F(\xi) = f(\xi^{2})$ for
some smooth function $f$ on $\mathbb{R}$ (cf.~\cite{Whi}), and we
get
$$ f^{(n)}(0) = \frac{n!}{(2n)!} F^{(2n)}(0) \ \ \mbox{ and } \ \
   \left| f^{(n)}(0) \right| \geq \frac{n! M'_{2n}}{(2n)!}, $$
which is the desired result.

\vspace{2ex}

We conclude this section with some examples, one of which (namely,
for $k=2$) will be applied to the construction of our
counterexample in the last section.

\vspace{1ex}

\begin{example}\label{example} Fix an integer $k \in \mathbb{N}$, $k \geq 1$, and put
$$ \log^{(k)} := \underbrace{\log \circ \ldots \circ \log}_{k
   \mbox{ \scriptsize times}}, \ \ \mbox{ and } \ \
   e\uparrow \uparrow k := (\underbrace{\exp \circ \ldots \circ \exp}_{k
   \mbox{ \scriptsize times}}) (1); $$
let $n_{k}$ be the smallest integer greater than $e \uparrow
\uparrow k$. Then the sequence
$$ \left( \log^{(k)} n \right)^{n} \ \ \mbox{ for } \ \ n \geq n_{k} $$
is log-convex. Further, the shifted sequence:
$$ M=(M_{n}), \ \ M_{n} := \frac{1}{\left( \log^{(k)} n_{k} \right)^{n_{k}} } \cdot
   \left( \log^{(k)} (n_{k} +n) \right)^{(n_{k}+n)}, $$
determines a quasianalytic class closed under derivatives; the
former follows from Cauchy's condensation criterion. It is easy to
check that the sequences $M^{(p)}$, $p>1$, are quasianalytic when
$k>1$, but are not quasianalytic when $k=1$.
\end{example}

\section{Non-extendability of quasianalytic germs}

In this section we are concerned with a result by
V.~Thilliez~\cite{Thil-2} on the extension of quasianalytic
function germs in one variable, recalled below. As before,
consider two log-convex sequences $M$ and $N$ with $M_{0}=N_{0}=1$
such that $\mathcal{Q}_{1}(M) \subset \mathcal{Q}_{1}(N)$. Denote
by $\mathcal{Q}_{1}(M)^{+}$ the local ring of right-hand side
germs at zero (i.e.\ germs of functions from
$\mathcal{Q}([0,\varepsilon],M)$ for some $\varepsilon >0$).

\vspace{2ex}

\begin{theorem}\label{non1}
If \ $\mathcal{Q}_{1}(N)$ is a quasianalytic local ring, then
$$ \mathcal{O}_{1} \varsubsetneq \mathcal{Q}_{1}(M) \subset
   \mathcal{Q}_{1}(N) \ \ \Longrightarrow \ \
   \mathcal{Q}_{1}(M)^{+} \setminus \mathcal{Q}_{1}(N) \neq \emptyset, $$
i.e.\ there exist right-hand side germs from
$\mathcal{Q}_{1}(M)^{+}$ which do not extend to germs from
$\mathcal{Q}_{1}(N)$. Here $\mathcal{O}_{1}$ stands for the local
ring of analytic function germs in one variable at $0 \in
\mathbb{R}$.
\end{theorem}


\begin{remark}\label{rem3} Theorem \ref{non1} may be related to
the research by M.~Langenbruch~\cite{Lb} on the extension of
ultradifferentiable functions in several variables, principally
focused on the non-quasianalytic case, which seems to be more
difficult in this context. His extension problem is, roughly
speaking, as follows:

\vspace{1ex}

\begin{em}
Given two compact convex subsets $K,K_{1}$ of $\mathbb{R}^{m}$
such that $\mbox{\rm int}\, (K) \neq \emptyset$ or $K = \{ 0 \}$
and $K \subset \mbox{\rm int}\, (K_{1})$, characterize the
sequences of positive numbers $M$ and $N$ such that every function
from the class $\mathcal{Q}(K,M)$ extends to a function from
$\mathcal{Q}(K_{1},N)$.
\end{em}
\end{remark}


M.~Langenbruch applies, however, different methods and techniques
in comparison with V.~Thilliez. In particular, his approach is
based on the theory of Fourier transform and plurisubharmonic
functions.

\vspace{1ex}

On the other hand, Thilliez's approach relies on Grothendieck's
version of the open mapping theorem (cf.~\cite{Gro}, Chap.~4,
Part~1, Theorem~2 or \cite{M-V}, Part~IV, Chap.~24) and Runge
approximation. It also enables a formulation of the
non-extendability theorem for quasianalytic function germs on a
compact convex subset $K \subset \mathbb{R}^{m}$ with $0 \in K$.

\vspace{1ex}

Nevertheless, in order to construct our counterexample in the next
section, we need a refinement of Theorem \ref{non1}, stated below.
Thilliez's proof can be adapted mutatis mutandis. We shall outline
it for the reader's convenience. Consider an increasing countable
family $M^{[p]}$, $p\in \mathbb{N}$, of log-convex sequences,
i.e.\
$$ 1= M^{[p]}_{0} \leq M^{[p]}_{1} \leq M^{[p]}_{2} \leq M^{[p]}_{3} \leq
   \ldots\ \ \ \mbox{ and } \ \ M^{[p]}_{j} \leq M^{[q]}_{j} $$
for all $j,p,q \in \mathbb{N}$, $p \leq q$. Then we receive an
ascending sequence of local rings
$$ \mathcal{Q}_{1}(M^{[1]}) \subset \mathcal{Q}_{1}(M^{[2]}) \subset
   \mathcal{Q}_{1}(M^{[3]}) \subset \ldots $$
such that $\mathcal{Q}_{1}(M^{[p]})$ is dominated by
$\mathcal{Q}_{1}(M^{(q)})$ for all $p,q \in \mathbb{N}$ with $p
\leq q$.

\vspace{2ex}

\begin{theorem}\label{non2}
If every local ring $\mathcal{Q}_{1}(M^{[p]})$ is quasianalytic,
then
$$ \mathcal{O}_{1} \varsubsetneq \mathcal{Q}_{1}(M) \subset \bigcup_{p \in \mathbb{N}} \,
   \mathcal{Q}_{1}(M^{[p]}) \ \ \Longrightarrow \ \
   \mathcal{Q}_{1}(M)^{+} \setminus \bigcup_{p \in \mathbb{N}} \,
   \mathcal{Q}_{1}(M^{[p]}) \neq \emptyset. $$
\end{theorem}


\begin{proof}
We adopt the following notation. For a smooth function $f$ on an
interval $I \subset \mathbb{R}$ and $r>0$, put
$$ \| f \|_{M,I,r} := \sup \, \left\{ \frac{\left| f^{(n)}(x) \right|}{r^{n} n!
   M_{n}}: \ n \in \mathbb{N}, \ x \in I \right\}. $$
For $k \in \mathbb{N}$, $k>0$, let \ $ B_{k}(M)$ \ or \
$B_{k}(M)^{+}$, respectively, denote the Banach space with norm
$$ \| \cdot \|_{M,[-1/k,1/k],k}  \ \ \mbox{ or } \ \ \| \cdot \|_{M,[0,1/k],k}, $$
of those smooth functions on the interval $[-1/k,1/k]$ or
$[0,1/k]$) such that
$$ \| f \|_{M,[-1/k,1/k],k} < \infty \ \ \mbox{ or } \ \
  \| f \|_{M,[0,1/k],k} < \infty, \ \ \mbox{respectively.} $$

As the canonical embeddings
$$ B_{k}(M) \hookrightarrow B_{l}(M) \ \ \mbox{ and } \ \
   B_{k}(M)^{+} \hookrightarrow B_{l}(M)^{+}, \ \ k,l \in \mathbb{N}, \ k \leq l, $$
are compact linear operators (a consequence of Ascoli's theorem;
cf.~\cite{Kom0}), one can endow the local rings
$\mathcal{Q}_{1}(M)$ and $\mathcal{Q}_{1}(M)^{+}$ with the
inductive topologies. Similarly, the countable union of local
rings
$$ \bigcup_{p \in \mathbb{N}} \, \mathcal{Q}_{1}(M^{[p]}) $$
is the inductive limit of the sequence $B_{k}(M^{(k)})$, $k \in
\mathbb{N}$, $k>0$, of Banach algebras. Further, the local ring
$$ \bigcup_{p \in \mathbb{N}} \, \mathcal{Q}_{1}(M^{[p]}) \cap
   \mathcal{Q}_{1}(M)^{+} $$
is the inductive limit of the sequence
$$ B_{k}(M^{(k)}) \cap B_{k}(M)^{+}, \ \ \ k \in \mathbb{N}, \ k>0, $$
of Banach algebras with norms
$$ \| f \|_{k} := \max \, \{ \| f \|_{M^{[k]},[-1/k,1/k],k}\, , \
   \| f \|_{M,[0,1/k],k} \}. $$
Clearly, the restriction operator
$$ R: \bigcup_{p \in \mathbb{N}} \, \mathcal{Q}_{1}(M^{[p]}) \cap
   \mathcal{Q}_{1}(M)^{+} \longrightarrow \mathcal{Q}_{1}(M)^{+} $$
is continuous and injective by quasianalyticity. We must show that
$R$ is not surjective.

\vspace{1ex}

To proceed with {\em reduction ad absurdum}, suppose that $R$ is
surjective. By Grothendieck's version of the open mapping theorem
(cf.~\cite{Gro}, Chap.~4, Part~1, Theorem~2 or \cite{M-V},
Part~IV, Chap.~24), the operator $R$ is a homeomorphism onto the
image. Further, by Grothendieck's factorization theorem ({\em
loc.~cit.\/}), for each $k \in \mathbb{N}$ there is an $l \in
\mathbb{N}$ and a constant $C>0$ such that
$$ R \left( B_{l}(M^{(l)}) \right) \supset R \left( B_{l}(M^{(l)}) \cap
   B_{l}(M)^{+} \right) \supset B_{k}(M)^{+}, $$
and
$$ \| R^{-1} f \|_{M^{(l)},[-1/l,1/l],l} \: \leq \:
   \| R^{-1} f \|_{l} \: \leq \: C \| f \|_{M,[0,1/k],k} $$
for all $f \in B_{k}(M)^{+}$. In particular, there is an $l \in
\mathbb{N}$ and a constant $A>0$ such that
$$ R \left( B_{l}(M^{(l)}) \right) \supset B_{1}(M)^{+}, $$
and
$$ \| R^{-1} f \|_{M^{(l)},[-1/l,1/l],l} \: \leq \:
   A \| f \|_{M,[0,1],1} $$
for all $f \in B_{1}(M)^{+}$. In particular,
$$ \left| P\left(-\frac{1}{l}\right) \right| \leq A \| P
   \|_{M,[0,1],1} $$
for every polynomial $P \in \mathbb{C}[x]$. Put
$$ W := \left\{ z \in \mathbb{C}: \: \mbox{dist}\, (z,[0,1]) \leq \frac{1}{2l} \right\}
$$
and
$$ B := \sup \, \left\{  \frac{(2l)^{n}}{M_{n}}: \: n \in
   \mathbb{N} \right\} < \infty; $$
the last inequality holds because \ $\mathcal{O}_{1} \varsubsetneq
\mathcal{Q}_{1}(M)$ \ whence
$$ \sup \, \left\{  \sqrt[n]{M_{n}}: \  n \in \mathbb{N} \right\} = \infty. $$
It follows from Cauchy's inequalities that
$$ \sup \, \left\{ \left| P^{(n)}(x) \right| : x \in [0,1] \right\} \leq n!
   (2l)^{n} \sup \, \{ |P(x)|: x \in W \}, $$
and hence
$$ \| P \|_{M,[0,1],1} \leq B \sup \, \{ |P(x)|: x \in W \}.  $$
Consequently,
$$ \left| P\left(-\frac{1}{l}\right) \right| \leq AB \sup \, \{ |P(x)|: x \in W \} $$
for every polynomial $P \in \mathbb{C} [x]$. But, by virtue of
Runge approximation, there exists a sequence of polynomials
$P_{\nu} \in \mathbb{C}[x]$ which converges uniformly to $0$ on
$W$, and to $1$ for $x=-1/l$. This contradicts the above estimate,
and thus the theorem follows.
\end{proof}

\section{Construction of a counterexample}

We give a counterexample indicating that quasianalytic structures,
unlike the classical structure $\mathbb{R}_{an}$, may not admit
quantifier elimination in the language augmented merely by the
reciprocal function $1/x$. The example we construct is a plane
curve through $0 \in \mathbb{R}^{2}$ which is definable in the
quasianalytic structure corresponding to the log-convex sequence
$$ M=(M_{n}), \ \ M_{n} := \frac{1}{\left( \log \log \, 3 \right)^{3} } \cdot
   \left( \log \log \, (n +3) \right)^{(n+3)}; $$
this sequence determines a quasianalytic class closed under
derivatives (cf.~example \ref{example}). By Theorem \ref{non2}, we
can take a function germ
$$ f \in \mathcal{Q}_{1}(M)^{+} \, \setminus \, \bigcup_{p \mbox{ \scriptsize odd}} \,
   \mathcal{Q}_{1}(M^{(p)}). $$
Let $V \subset \mathbb{R}^{2}$ be the graph of a representative of
this germ in a right-hand side neighbourhood $[0,\varepsilon]$.

\vspace{1ex}

To proceed with {\em reductio ad absurdum\/}, suppose $V$ is given
by a term in the language of restricted $\mathcal{Q}_{M}$-analytic
functions augmented merely by the reciprocal function $1/x$.
Taking into account Remark~\ref{rem-2.1}, we can thus deduce from
Theorem~\ref{rect-th} that there would exist a rectilinearization
of this term by a finite sequence of blow-ups of the real plane at
points. Consequently, the germ of $V$ at zero would be contained
in the image $\varphi([-1,1])$, where
$$ \varphi =(\varphi_{1},\varphi_{2}): [-1,1]
   \longrightarrow \mathbb{R}^{2}, \ \ \varphi(0)=0, $$
is a $\mathcal{Q}_{M}$-analytic homeomorphism. But then the order
of $\varphi_{1}$ at zero must be odd, and thus the set $V$ would
have a parametrization near zero of the form $(\xi^{p}, g(\xi))$,
where $p$ is an odd positive integer and $g$ is a
$\mathcal{Q}_{M}$-function in the vicinity of zero. Hence and by
Theorem \ref{substitution}, we would get
$$ f(x) = g(x^{1/p}) \in \mathcal{Q}_{1}(M^{(p)}), $$
which is a contradiction.


\begin{remark} By virtue of Puiseux's theorem for definable functions
(cf.~\cite{Now5}, Section~2), the germ of every smooth function in
one variable that is definable in the structure
$\mathbb{R}_{\mathcal{Q}_{M}}$ belongs to
$\mathcal{Q}_{1}(M^{(p)})$ for some positive integer $p$.
Therefore the structure under study will not admit quantifier
elimination, even considered with the richer language of
restricted definable quasianalytic functions augmented by the
reciprocal function $1/x$.
\end{remark}

\begin{remark}
Also, our counterexample demonstrates that the classical theorem
of \L{}ojasiewicz~\cite{Loj} that every subanalytic set of
dimension $\leq 1$ is semianalytic is no longer true in
quasianalytic structures.
\end{remark}

\vspace{2ex}

{\bf Acknowledgements.} This research was partially supported by
Research Project No.\ N N201 372336 from the Polish Ministry of
Science and Higher Education.

\vspace{5ex}

\begin{small}
Institute of Mathematics

Faculty of Mathematics and Computer Science

Jagiellonian University

ul.~Profesora \L{}ojasiewicza 6

30-348 Krak\'{o}w, Poland

e-mail:
\email{\textcolor[rgb]{0.00,0.00,0.84}{nowak@im.uj.edu.pl}}

\end{small}

\end{document}